\documentclass{amsart}
\usepackage{amsthm,amsmath,amssymb,mathtools,mathrsfs}
\usepackage{hyperref}
\usepackage{enumitem}
\usepackage{chngcntr}
\usepackage{apptools}
\usepackage{accents}
\usepackage[foot]{amsaddr}

\newcommand{\R}{\mathbb{R}} 
\newcommand{\N}{\mathbb{N}} 
\newcommand{\normal}{\mathcal{N}}
\newcommand{\dint}{{\rm d}}
\renewcommand{\P}{\mathbb{P}}
\newcommand{\E}{\mathbb{E}}
\renewcommand{\epsilon}{\varepsilon}
\def\B{\mathcal{B}}

\newcommand{\argmax}{\arg\!\max}
\newcommand{\Poi}{\text{Poi}}

\newcommand{\abs}[1]{\left\vert #1 \right\vert}

\newtheorem{thm}{Theorem}
\newtheorem{coro}{Corollary}
\newtheorem{lemma}{Lemma}

\newtheorem{remark}{Remark}
\newtheorem{example}{Example}

\theoremstyle{definition}
\newtheorem{Def}{Definition}

\usepackage{color}
\definecolor{darkred}{RGB}{139,0,0}
\definecolor{darkgreen}{RGB}{0,100,0}
\definecolor{darkmagenta}{RGB}{139,0,139}
\definecolor{darkpurple}{RGB}{110,0,180}
\definecolor{darkblue}{RGB}{40,0,200}
\definecolor{darkorange}{RGB}{255,140,0}

\newcommand{\md}[1]{\textcolor{darkgreen}{#1}}

\allowdisplaybreaks

\begin{document}

\title[Maximum likelihood 
estimation in HMMs with inhomogeneous noise]{Maximum likelihood 
estimation in hidden Markov models
with inhomogeneous noise}

\author{Manuel Diehn$^1$, Axel Munk$^{1,2,3}$, Daniel Rudolf$^{1,3}$}

\address{$^1$Institute for Mathematical Stochastics, 
Georg-August-University of G\"ottingen, Goldschmidtstra\ss e 7, 37077 G\"ottingen;
}
\address{$^2$Max Planck Insititute for Biophysical Chemistry, Am Fa\ss berg 11, 
 37077 G\"ottingen}
\address{$^3$Felix-Bernstein-Institute for Mathematical Statistics 
 in the Biosciences, Goldschmidtstra\ss e 7, 37077 G\"ottingen}
 
\email{mdiehn1@gwdg.de\ \&\ amunk1@gwdg.de\ \&\ daniel.rudolf@uni-goettingen.de}



\begin{abstract}
We consider parameter estimation in 
finite hidden state space Markov models with time-dependent inhomogeneous noise,
where the inhomogeneity vanishes sufficiently fast.
Based on the concept of asymptotic mean stationary processes 
we prove 
that the maximum likelihood and 
a quasi-maximum likelihood estimator (QMLE) are  strongly consistent.
 The computation of the QMLE 
 ignores the inhomogeneity, hence, is much simpler and robust.
The theory is motivated by an example from biophysics and applied to a Poisson- and linear Gaussian model.
 \end{abstract}

\maketitle 
 
\section{Introduction} 
\noindent
\emph{Motivation.}
Hidden Markov models (HMMs) have a long history and are widely used in
a plenitude of applications  ranging from econometrics, chemistry, biology, 
speech recognition to neurophysiology. For example, 
transition rates between openings and closings 
of ion channels, see \cite{neher_patch_1992}, are often assumed to be Markovian and the
observed conductance levels from such experiments can be
modeled with homogeneous HMMs.
The HMM is typically justified if the underlying experimental conditions,
such as the applied voltage in ion channel recordings,
are kept constant over time, 
see \cite{hotz_idealizing_2013,qin_hidden_2000,vandongen_new_1996,
siekmann_mcmc_2011,venkataramanan_applying_2002}. 

However, if the conductance levels are measured  in experiments
with varying voltage over time, then the noise
appears to be inhomogeneous, i.e., the noise 
has a voltage-dependent component. 
Such experiments play an important 
role in the understanding of the dependence of
the gating behavior to the gradient of the applied voltage 
\cite{briones_voltage_2016,danelon_interaction_2006}.	
To the best of our knowledge, there is
a lack of a rigorous statistical methodology for analyzing
such type of problems, for which we provide some first theoretical insights.
More detailed, in this paper we are concerned with 
the consistency of the maximum likelihood estimator (MLE) in such models and 
with the question of how
much maximum likelihood estimation in a homogeneous model is affected by 
inhomogeneity of the noise, a problem which appears to be relevant to many other situations, as well.

A \emph{homogeneous hidden Markov model}, as considered in this paper, is given by a bivariate 
stochastic process \mbox{$(X_n,Y_n)_{n\in\mathbb{N}}$}, 
where $(X_n)_{n\in \mathbb{N}}$
is a Markov chain with finite state space $S$, and $(Y_n)_{n\in \mathbb{N}}$ is, 
conditioned on $(X_n)_{n\in\mathbb{N}}$,
an independent sequence of random variables mapping to a Polish space $G$, 
such that the distribution of $Y_n$ depends only on $X_n$. 
The Markov chain $(X_n)_{n\in\mathbb{N}}$ is not observable, but observations of $(Y_n)_{n\in \mathbb{N}}$
are available. 
A well known statistical method to estimate the unknown parameters is based on the maximum likelihood principle, see \cite{baum_statistical_1966,baum_maximization_1970}.
The study of consistency and asymptotic normality of 
the MLE of such homogeneous 
HMMs has a long history and is nowadays
well understood in quite general situations. 
We refer to the final paragraph of this section for a review but 
already mention that the approach of 
\cite{douc_consistency_2011} is particularly useful for us.

In contrast to the classical setting, we consider an \emph{inhomogeneous HMM}, 
namely a bivariate stochastic process $(X_n,Z_n)_{n\in\mathbb{N}}$, where
conditioned on $(X_n)_{n\in \mathbb{N}}$ we assume that 
$(Z_n)_{n\in \mathbb{N}}$ is a sequence of independent random variables 
on space $G$, such that the distribution
of $Z_n$ depends not only on the value of $X_n$, but additionally on $n\in \mathbb{N}$. 
The dependence on $n$ implies that the
Markov chain $(X_n,Z_n)_{n\in\mathbb{N}}$ is inhomogeneous. 
In such generality a theory for
maximum likelihood estimation in 
inhomogeneous hidden Markov models is, of course, a 
notoriously difficult task.

However, motivated by the example above (for details see below) 
we consider a specific situation
where e.g. the inhomogeneity is caused by an exogenous quantity (e.g. the varying voltage) 
with decreasing influence as $n$ increases .
To this end, we introduce the concept of a \emph{doubly hidden Markov model} (DHMM).
\begin{Def}[DHMM]\label{def: DHMM}
A doubly hidden Markov model is a trivariate 
stochastic process $(X_n,Y_n,Z_n)_{n\in\mathbb{N}}$
such that $(X_n,Y_n)_{n\in\mathbb{N}}$
is a non-observed homogeneous HMM and $(X_n,Z_n)_{n\in \mathbb{N}}$
is an inhomogeneous HMM with observations $(Z_n)_{n\in\N}$.
\end{Def}
For such a DHMM we have in mind that the 
distribution of $Z_n$ is getting ``closer'' to 
the distribution of $Y_n$ for increasing $n$.
A crucial point here is that $(Z_n)_{n\in \mathbb{N}}$
is observable whereas $(Y_n)_{n\in \mathbb{N}}$ is not. 
Because of the ``proximity'' of $Z_n$ and $Y_n$ one might hope 
to carry theoretical results from homogeneous HMMs to inhomogeneous ones.

We illustrate a setting of a DHMM by modeling the conductance level of ion channel data with
varying voltage\footnote{Measurements are kindly provided by the lab of C.~Steinem, Institute for Organic and 
Molecular Biochemistry, University of G\"ottingen}.
In Figure~\ref{example ramp} measurements of
the current flow across the outer cell membrane of the
 porin PorB of Neisseria meningitidis are displayed in order to investigate the antibacterial resistance of the PorB channel.
As the applied voltage $(u_n)_{n\in\mathbb{N}}$  increases linearly 
Ohm's law suggests that the measured current increases also
linearly, see Figure~\ref{example ramp}. A reasonable model for 
the observed current is to assume that  
it follows a Gaussian hidden Markov model,
i.e., the dynamics can be described by
\begin{equation} \label{dyn}
u_n(\mu^{(X_n)} + \sigma^{(X_n)}V_n)+\tilde{\epsilon}_n.
\end{equation}
Here the observation space $G=\mathbb{R}$
and the finite state space of the hidden Markov chain $(X_n)_{n\in \mathbb{N}}$ 
is assumed to be 
$S=\{1,2\}$, which corresponds to an ``open'' and ``closed'' gate.
For $i=1,2$, the expected slope is $\mu^{(i)}\in \mathbb{R}$, 
the noise level $\sigma^{(i)}\in(0,\infty)$
and $(V_n)_{n\in\mathbb{N}}$ is an i.i.d. standard normal sequence, i.e.,
$V_1\sim \normal(0,1)$, where $\mathcal{N}(\mu,\sigma^2)$
 denotes the
 normal distribution with mean $\mu\in \mathbb{R}$ and variance $\sigma^2>0$.
Further, $(\tilde{\epsilon}_n)_{n\in\N}$
 is another sequence of real-valued 
i.i.d. random variables, independent of $(V_n)_{n\in\N}$, 
with $\tilde{\epsilon}_1\sim \normal(0,\kappa^2)$ and $\kappa^2>0$, 
which is necessary to model the background noise, even when $u_n=0$.

Dividing the dynamic \eqref{dyn} by $u_n$ gives the conductivity
of the channel, see Figure~\ref{example conductivity}.
\begin{figure}[!htb]
	\centering
  \includegraphics[width=1\textwidth]{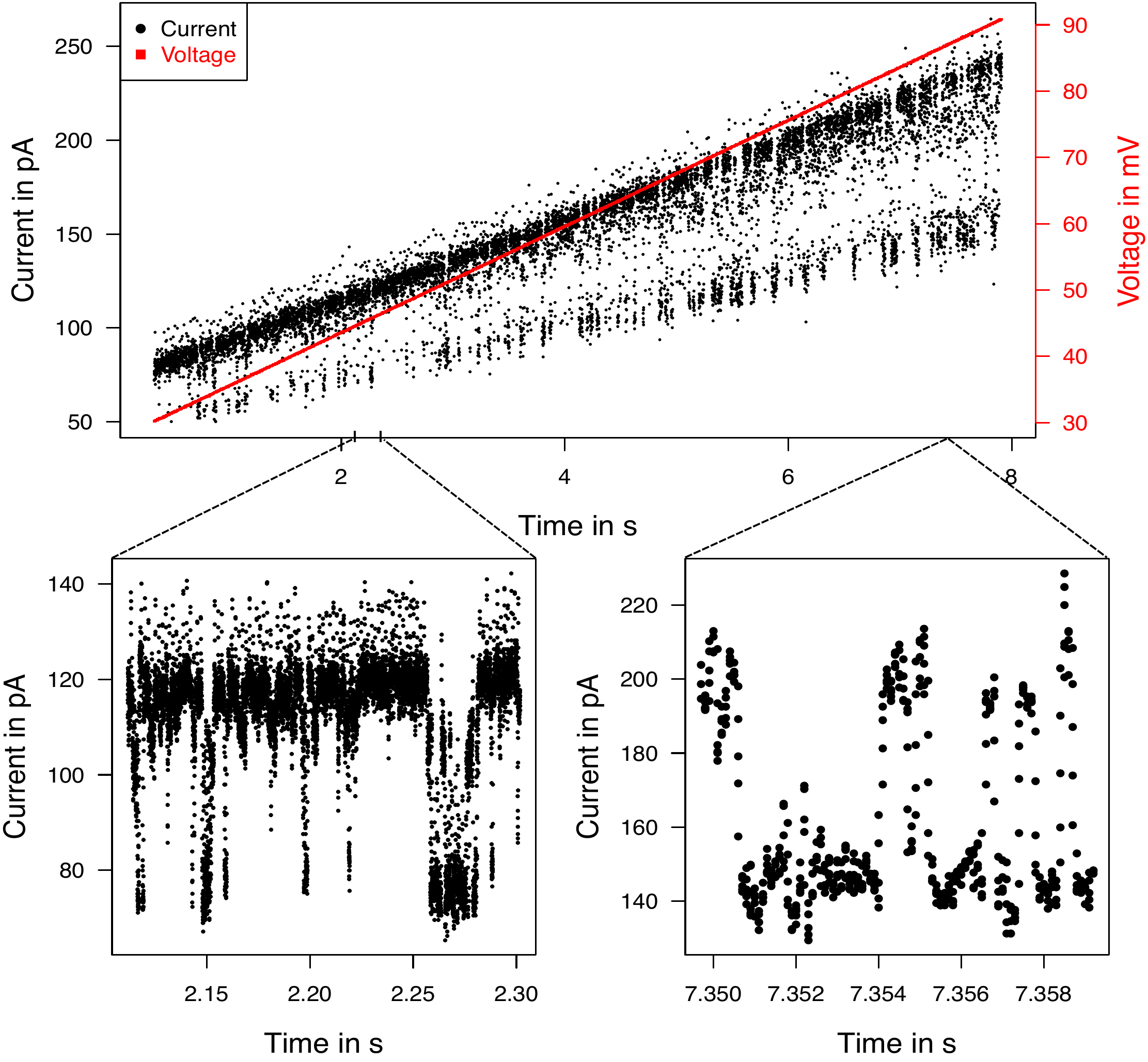}
	\caption{ Above: Measurements at a large time scale (seconds)
	of the current flow of a PorB mutant protein driven by linear increasing voltage 
	from $30mV$-$120mV$.Below: Zoom into finer time scales (decisecond to millisecond).
	}
	\label{example ramp}
\end{figure}
\begin{figure}[!htb]
	\centering
  \includegraphics[width=1\textwidth]{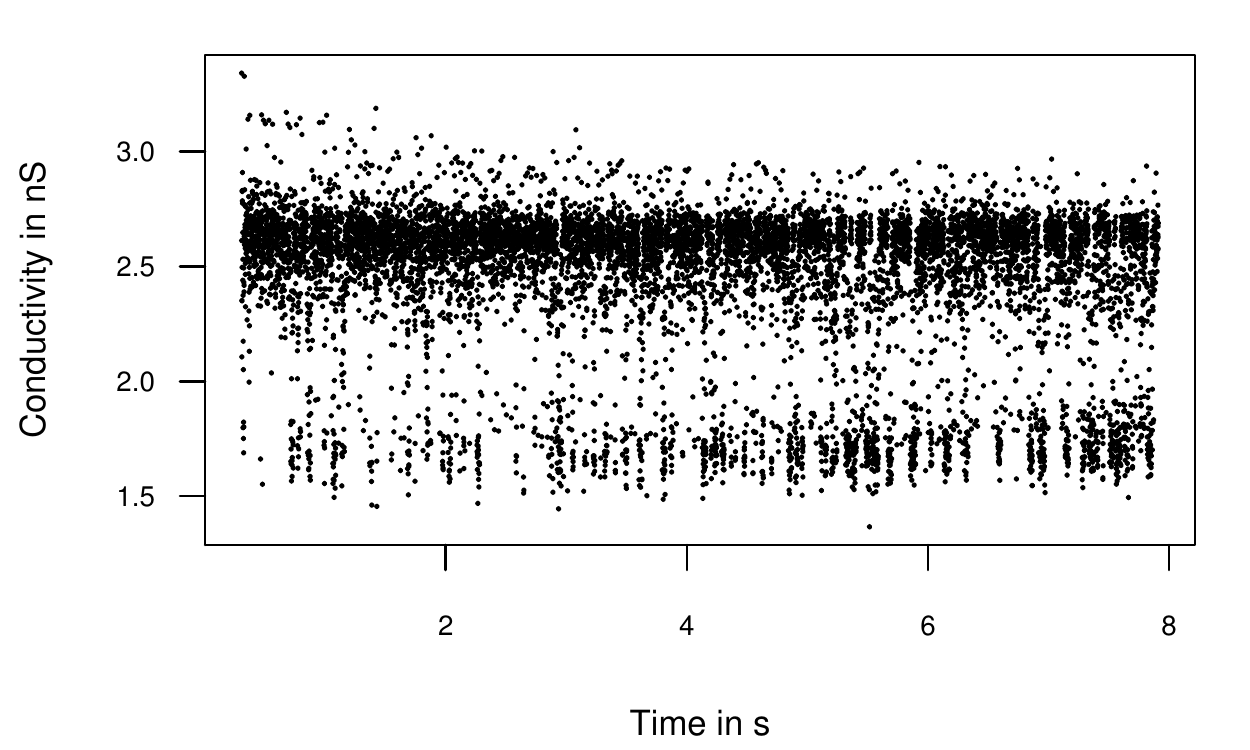}
	\caption{Conductivity of the protein PorB. The variance of
	the data decreases in time to a constant.}
	\label{example conductivity}
\end{figure}
This is now a sequence $(Z_n)_{n\in\mathbb{N}}$ of an inhomogeneous HMM.
The state of the Markov chain determines the parameter
$(\mu^{(1)},\sigma^{(1)})$ 
or $(\mu^{(2)},\sigma^{(2)})$, both unknown. 
The non-observable sequence of 
random variables $(Y_n)_{n\in \mathbb{N}}$
of the homogeneous HMM is given by
\begin{align}  \label{eq: Y_ion}
    Y_n & := \mu^{(X_n)} + \sigma^{(X_n)} V_n.  
 \end{align}
 The observation $(Z_n)_{n\in\mathbb{N}}$ of the inhomogeneous HMM
 is determined by 
 \begin{equation}  \label{eq: Z_ion}
  Z_n := Y_n + \varepsilon_n,
  \end{equation}
with $\varepsilon_n= \tilde{\epsilon}_n/u_n$, such that  
$\varepsilon_n\sim \mathcal{N}(0,\beta_n^2)$ where $\beta_n=\kappa/u_n$ and 
$\lim_{n\to \infty} \beta_n^2 =0,$ as the voltage increases.
Such a DHMM describes approximately the observed conductance level of ion channel 
 recordings with linearly increasing voltage.

 Intuitively, here one can already see that for sufficiently large $n$ the influence of  $\varepsilon_n$ ``washes out'' as 
 $\beta_n$  decreases to zero  and observations of $Z_n$ are ``close'' to $Y_n$. 

\emph{Main result.} We explain now our main theoretical contribution for such a DHMM. Assume 
that we have a parametrized DHMM $(X_n,Y_n,Z_n)_{n\in\N}$ 
 with compact parameter space $\Theta\subseteq \mathbb{R}^d$.
For $\theta\in \Theta$ let 
$q_\theta^\nu$ be the likelihood function of $Y_1,\ldots,Y_n$ and
$p_\theta^\nu$ be the likelihood function of $Z_1,\ldots,Z_n$ with $X_0\sim \nu$.
Both functions  are assumed to be continuous in $\theta$.
Given observations $z_1,\dots,z_n$ of $Z_1,\dots,Z_n$ our goal is to estimate
``the true'' parameter $\theta^*\in \Theta$. 
The MLE $\theta_{\nu,n}^{\rm\,ML}$, given by a parameter in the set of 
maximizers of the log-likelihood function, i.e.,
\[
 \theta_{\nu,n}^{\rm\,ML} \in \argmax_{\theta\in \Theta}\log p_\theta^\nu(z_1,\dots,z_n),
\]
is the canonical estimator for approaching this problem.  
Note that this set is non-empty due to the compactness of the parameter space
and the continuity of $p^{\nu}_{\theta}(z_1,\ldots,z_n)$ in $\theta$.
Unfortunately none of the strong consistency results of 
maximum likelihood parameter estimation 
provided for homogeneous HMMs
are applicable, because of the inhomogenity.
Namely, all proofs for consistency in HMMs
rely on the fact that the conditional distribution of $Z_n $ given
$X_n = x$ is constant for all $n\in\N$.
In a DHMM this is usually not the case 
for $(Z_n)_{n\in\mathbb{N}}$, because of the time-dependent noise. 
This issue can be circumvented by proving 
that under suitable assumptions $(Z_n)_{n\in \mathbb{N}}$ is an \emph{asymptotic
mean stationary process}. This implies 
ergodicity and an ergodic theorem for $(Z_n)_{n\in\mathbb{N}}$, that can be used. 
However, for the computation of $\theta_{\nu,n}^{\rm\,ML}$ explicit knowledge of the
inhomogeneity is needed, i.e., of the time-dependent component of the noise 
which is hardly known in practice (recall our data example).
That is the reason for us to introduce a quasi-maximum likelihood estimator (QMLE), 
given by a maximizer of the quasi-likelihood function, i.e.,
\[
 \theta_{\nu,n}^{\rm\,QML} \in \argmax_{\theta\in \Theta} \log q_\theta^\nu(z_1,\dots,z_n).
\]
This is not a MLE, since the observations 
are generated from the inhomogeneous model,
whereas  $q_{\theta}^{\nu}$ is the likelihood function of the homogeneous model.
Roughly, we assume the following (for a precise definition see Section~\ref{sec: struct_cond}):
\begin{enumerate}[label=\arabic*.)]
 \item The transition matrix of the hidden finite state space Markov chain is irreducible and
  satisfies a continuity condition w.r.t. the parameters.
 \item\label{vague closeness} The observable and non-observable random variables $(Z_n)_{n\in\N}$
 and $(Y_n)_{n\in\N}$ are ``close'' to each other in a suitable sense.
 \item The homogeneous HMM is well behaving, such that observations of $(Y_n)_{n\in\N}$
  would lead to a consistent MLE.
\end{enumerate}
We show that if the $Z_n$ approximate the $Y_n$ 
reasonably well (see the condition \ref{en: C1} in Section~\ref{sec: struct_cond} )
the estimator $\theta_{\nu,n}^{\rm\,QML}$
provides also a reasonable way for approximating ``the true'' parameter $\theta^*$.
If the model satisfies all conditions, see Section~\ref{sec: struct_cond}, 
then Theorem~\ref{thm: main_thm}, states that
\[
 \theta_{\nu,n}^{\rm\,QML} \to \theta^*\quad \text{a.s.},\text{ as } n\to\infty.
\]
Hence the QMLE is consistent. 
As a consequence we obtain under an additional assumption that also 
the MLE is consistent,
$\theta_{\nu,n}^{\rm\,ML} \to \theta^*$
almost surely, as $n\to \infty$. 
For a Poisson model 
and linear Gaussian model we specify Theorem~\ref{thm: main_thm}, see Section~\ref{sec: appl}. 
In the DHMM described in \eqref{eq: Y_ion} and \eqref{eq: Z_ion}
we obtain consistency of the QMLE
whenever $\beta_n=\mathcal{O}(n^{-q})$ for some $q>0$.
In Section~\ref{sec: disc} we reconsider the approximating condition \ref{vague closeness}, 
precisely stated in Section~\ref{sec: struct_cond}, provide an outlook to possible extensions 
and discuss asymptotic normality of the estimators.

\textit{Literature review and connection to our work.} The study of maximum likelihood estimation in
homogeneous hidden Markov models has a long history and was initiated by Baum and 
Petrie, see \cite{baum_statistical_1966,baum_maximization_1970}, 
who
proved strong consistency of 
the MLE for finite state spaces $S$ and $G$.
Leroux extends this result to general observation spaces in \cite{leroux_maximum-likelihood_1992}. 
These consistency results rely on ergodic theory for stationary processes 
which is not applicable in our setting since
the process we observe is not stationary. More precisely, it was shown that the 
relative entropy rate converges for any parameter $\theta$ in the parameter space $\Theta$
using an ergodic theorem for subadditive processes.
There are further extensions also to Markov chains on general 
state spaces, but under stronger assumptions, see \cite{douc_asymptotics_2001,douc_asymptotic_2004,genon-catalot_lerouxs_2006,legland_asymptotic_1997,heymans_consistent_1986}.
A breakthrough has been achieved by Douc et al. \cite{douc_consistency_2011} who used the concept of exponential
separability. This strategy allows one to bound the relative entropy rate directly. 

Although the state space of the Markov chain is more general than
in our setting, we cannot apply the results of 
\cite{douc_consistency_2011}
due to the inhomogeneity of the observation,
but we use the same approach to show our consistency statements.

The investigation of strong consistency of maximum likelihood estimation 
in inhomogeneous HMMs is less developed. In \cite{ailliot_consistency_2015} 
and \cite{pouzo_maximum_2016}  the
MLE in inhomogeneous
Markov switching models is studied. 
There, the transition probabilities are also influenced by the observations, but the inhomogeneity there
is different from the time-dependent inhomogeneity considered in our work, since the conditional law
is not changing over time.

Related to strong consistency, as considered here, is the investigation of asymptotic
normality (as it provides weak consistency). For homogeneous HMMs asymptotic 
normality has be shown for example
in \cite{douc_asymptotic_2004,bickel_asymptotic_1998}.
In \cite{pouzo_maximum_2016}, also, asymptotic normality 
for the MLE in
Markov switching models is studied whereas in
\cite{jensen_asymptotic_2011} asymptotic normality of \emph{M}-estimators
in more general inhomogeneous situations is considered.
However, the QMLE we suggest and analyze
does not satisfy the assumptions 
imposed there.
In Section~\ref{QMLE_asymp_nornal} and in Appendix~\ref{Ass_asymp_norm} we provide and 
discuss necessary conditions to achieve asymptotic normality
for the QMLE by adapting the approach of \cite{jensen_asymptotic_2011}.

To ease readability Section~\ref{sec: main_thm_proof} is devoted to the proofs of our main results. 
In particular, we draw 
the connection between asymptotic mean stationary processes 
and inhomogeneous hidden Markov models.

\section{Setup and notation}
\noindent
We denote the finite state space of $(X_n)_{n\in\N}$ by 
$S=\{1,\dots,K\}$ and $\mathcal{S}$ denotes the power set of $S$.
Furthermore, let $(G,m)$ be a Polish space with metric $m$ and corresponding
Borel $\sigma$-field $\mathcal{B}(G)$. 
The measurable space $(G,\mathcal{B}(G))$ is equipped with a $\sigma$-finite 
reference measure $\lambda$.
Througout the whole work we consider parametrized families of 
DHMMs (see Definition \ref{def: DHMM}) with compact parameter space $\Theta\subset \mathbb{R}^d$ for some $d\in\mathbb{N}$.
For this let 
$(\mathbb{P}_\theta)_{\theta\in \Theta}$ be a sequence of probability measures on 
a measurable space $(\Omega, \mathcal{F})$ such that for each 
parameter $\theta$ the distribution of $(X_n,Y_n,Z_n)\colon (\Omega,\mathcal{F},\mathbb{P}_\theta)\to S\times G\times G$
is specified by
\begin{itemize}
 \item an initial distribution $\nu$ on $S$ and a $K\times K$ transition matrix 
 $P_\theta=(P_\theta(s,t))_{s,t\in S}$
 of the Markov chain $(X_n)_{n\in\mathbb{N}}$,
such that
\[
 \mathbb{P}_\theta (X_n=s) =  
 \nu P_\theta^{n-1}(s), \quad s\in S, 
\]
 where $\nu P_\theta^0 = \nu $ and for $n>1$,
 \[
  \nu P_\theta^{n-1}(s) = \sum_{s_1,\dots,s_{n-1}\in S} P_\theta(s_{n-1},s) \prod_{i=1}^{n-2} P_\theta(s_i,s_{i+1})  \nu(s_1),
  \quad s\in S;
 \]
 (Here and elsewhere we use the convention
 that $\prod_{i=1}^0 a_i = 1$ for any sequence $(a_i)_{i\in \N}\subset \R$.)
 \item 
 and by the conditional distribution $Q_{\theta,n}$ 
 of $(Y_n,Z_n)$ given $X_n=s$, that is, 
 \[
  \mathbb{P}_\theta ((Y_n,Z_n)\in C\mid X_n=s) = Q_{\theta,n}(s,C),
  \qquad C\in \mathcal{B}(G^2)
 \]
 which satisfies that there are conditional density functions 
 $f_\theta, f_{\theta,n} \colon S\times G \to [0,\infty)$ w.r.t. $\lambda$, such that 
   \begin{align*}
    \mathbb{P}_\theta(Y_n\in A\mid X_n=s) & 
  = Q_{\theta,n}(s,A\times G) = \int_A f_{\theta}(s,y) \lambda(\dint y), \quad A\in \mathcal{B}(G),\\
  \mathbb{P}_\theta(Z_n\in B\mid X_n=s)& =Q_{\theta,n}(s,G\times B) 
  = \int_B f_{\theta,n}(s,z) \lambda(\dint z),\; B\in \mathcal{B}(G).
 \end{align*}
  Here the distribution of $Y_n$ given $X_n=s$ is independent of $n$, whereas
 the distribution of $Z_n$ given $X_n=s$ depends through $f_{\theta,n}$ also explicitly on $n$.
\end{itemize}
\noindent 
By $\mathcal{P}(S)$ we denote the
set of probability measures on $S$.
To indicate the dependence on the initial distribution, say $\nu\in \mathcal{P}(S)$, 
 we write $\mathbb{P}_\theta^{\nu}$ instead of just $\mathbb{P}_\theta$. 
 To shorten the notation, let $X=(X_n)_{n\in \mathbb{N}}$, $Y=(Y_n)_{n\in \mathbb{N}}$
 and $Z=(Z_n)_{n\in \mathbb{N}}$. Further, let $\P^{\nu,Y}_{\theta}$
 and $\P^{\nu,Z}_{\theta}$ be the distributions of $Y$ and $Z$ on $(G^{\mathbb{N}},\B(G^{\mathbb{N}}))$, respectively.

%

The ``true'' underlying model  parameter will be denoted as $\theta^*\in \Theta$ and 
we assume that the transition matrix $P_{\theta^*}$ possesses a unique 
invariant distribution 
 $\pi\in\mathcal{P}(S)$. We have access to a finite length observation of
 $Z$.
 Then, the problem is to find a consistent 
 estimate of $\theta^{*}$
 on the basis of the observations
 without observing $(X_n,Y_n)_{n\in \mathbb{N}}$.
 Consistency of the estimator of $\theta^*$ is limited up to equivalence classes in
 the following sense. Two parameters $\theta_1,\theta_2\in\Theta$
 are equivalent, written as $\theta_1\sim \theta_2$, iff there exist two 
 stationary distributions $\mu_1,\mu_2\in\mathcal{P}(S)$ 
 for $P_{\theta_1}, P_{\theta_2}$, respectively,
 such that $\mathbb{P}^{\mu_1,Y}_{\theta_1} = \mathbb{P}^{\mu_2,Y}_{\theta_2}$.
For the rest of the work
assume that each $\theta\in\Theta$ represents
its equivalence class.

For an arbitrary finite measure $\nu$ on $(S,\mathcal{S})$, $t\in\N$, $x_{t+1}\in S$ and $z_1,\ldots,z_t\in G$ define
\begin{align*}
p^{\nu}_{\theta}(x_{t+1};z_1,\ldots,z_t) &\coloneqq \sum\limits_{x_1,\ldots,x_t \in S} \nu(x_1) \prod\limits_{i=1}^t f_{\theta,i}(x_i,z_i) P_{\theta}(x_i,x_{i+1}),\\
p^{\nu}_{\theta}(z_1,\ldots,z_t) &\coloneqq \sum\limits_{x_{t+1} \in S} p^{\nu}_{\theta}(x_{t+1};z_1,\ldots,z_t).
\end{align*} 
If $\nu$ is a probability measure on $(S,\mathcal{S})$, then
$p^{\nu}_{\theta}(z_1,\ldots,z_n)$
is the likelihood of the observations
 $(Z_1,\ldots,Z_n)=(z_1,\dots,z_n)\in G^n$ 
 for the inhomogeneous HMM $(X_n,Z_n)_{n\in \mathbb{N}}$ with parameter $\theta\in\Theta$
 and $X_1\sim \nu$.
 Although there are no observations of 
 $Y$
 available, we  
 define similar quantities for $(Y_1,\dots,Y_n)=(y_1,\dots,y_n)\in G^n$ by
 \begin{align*}
q^{\nu}_{\theta}(x_{t+1},y_1,\ldots,y_t) &\coloneqq \sum\limits_{x_1,\ldots,x_t \in S} \nu(x_1) \prod\limits_{i=1}^t f_{\theta}(x_i,y_i) P_{\theta}(x_i,x_{i+1}),\\
q^{\nu}_{\theta}(y_1,\ldots,y_t) &\coloneqq \sum\limits_{x_{t+1} \in S} q^{\nu}_{\theta}(x_{t+1},y_1,\ldots,y_t).
 \end{align*}

 \section{Assumptions and main result}
\label{sec: ass_and_main_res}
\noindent
Assume for a moment that observations $y_1,\dots,y_n$ of $Y_1,\dots,Y_n$
are available. Then the log-likelihood function of $q_\theta^\nu$, with initial
distribution $\nu\in \mathcal{P}(S)$, is given by
\[
 \log q_\theta^{\nu}(y_1,\dots,y_n).
\]
In our setting we do not have access to observations of $Y$, but have
access to ``contaminated'' observations $z_1,\dots,z_n$ of $Z_1,\dots,Z_n$. Based 
on these observations define
a quasi-log-likelihood function
\[
 \ell^{\,\rm Q}_{\nu,n}(\theta) := \log q_\theta^{\nu}(z_1,\dots,z_n),
\]
i.e., we plug the contaminated observations into the likelihood of $Y_1,\ldots,Y_n$.
Now we approximate $\theta^*$ by $\theta_{\nu,n}^{\,\rm QML}$
which is the QMLE,
that is,
\begin{equation}  \label{eq: max_plug_in_LE}
\theta_{\nu,n}^{\,\rm QML}\in\argmax_{\theta \in \Theta} \ell_{\nu,n}^{\,\rm Q}(\theta).
\end{equation}
In addition, we are interested in the ``true'' MLE
of a realization $z_1,\ldots,z_n$ of
 $Z_1,\ldots,Z_n$.
For this define the log-likelihood function
\[
 \ell_{\nu,n}(\theta) := \log p_\theta^{\nu}(z_1,\dots,z_n),
\]
which leads to the MLE $\theta_{\nu,n}^{\,\rm ML}$ given by
\begin{equation}  \label{eq: MLE}
 \theta_{\nu,n}^{\, \rm ML} \in \argmax_{\theta \in \Theta} \ell_{\nu,n}(\theta).
\end{equation}
Under certain structural assumptions we prove that 
the QMLE from \eqref{eq: max_plug_in_LE} 
is consistent.
By adding one more condition this result can be used to verify that the MLE from
\eqref{eq: MLE} is also consistent.

\subsection{Structural conditions} \label{sec: struct_cond}
\noindent
We prove consistency of the QMLE $\theta_{\nu,n}^{\rm\,QML}$ 
and the MLE $\theta_{\nu,n}^{\,\rm ML}$
under the following structural assumptions:
\subsubsection*{Irreducibility and continuity of $X$}
\begin{enumerate}[label=(P\arabic*)]
 \item\label{en: P1} The transition matrix $P_{\theta^*}$ is irreducible.
 \item\label{en: P2} 
 The parametrization $\theta \mapsto P_\theta$ is continuous.
\end{enumerate}

\subsubsection*{Proximity of $Y$ and $Z$}
\begin{enumerate}[label=(C\arabic*)]
\item\label{en: C1} There exists $p>1$ such that
for any $s\in S$ and $\varepsilon>0$ we have
\begin{align*}
\P_{\theta^*}\left(m(Z_n,Y_n)\geq\varepsilon\mid X_n = s\right) = \mathcal{O}(n^{-p}).
\end{align*}
(Recall that $m$ is the metric on $G$.)
\item\label{en: C2} There exists an integer $k\in\N$ such that
\begin{align}\label{eq: finite}
\P_{\theta^*}^{\pi}\left(\prod_{i=1}^{k-1} \max_{s\in S} \frac{f_{\theta^*,i}(s,Z_i)}{f_{\theta^*}(s,Z_i)}<\infty\right) &= 1,\\
\E_{\theta^*}^{\pi}\left[ \max\limits_{s'\in S}
\frac{f_{\theta^*,n}(s',Z_n)}{f_{\theta^*}(s',Z_n)}\mid X_n =s \right] 
&< \infty,\quad \forall s\in S, n\geq k, \notag
\end{align} 
and
\begin{align}
\label{al: expect}
  \limsup\limits_{n\rightarrow\infty} 
  \E_{\theta^*}^{\pi}\left[  
  \max\limits_{s'\in S}\frac{f_{\theta^*,n}(s',Z_n)}{f_{\theta^*}(s',Z_n)} \mid X_n =s \right] 
  \leq 1,\quad \forall s\in S.
\end{align}
\item\label{en: C3} For every $\theta\in \Theta$ with $\theta \not \sim \theta^*$, 
there exists a neighborhood $\mathcal{E}_{\theta}$ of $\theta$ such that there exists an integer $k\in\N$ with
\begin{align}\label{eq: theta finite}
\P_{\theta^*}^{\pi}
\left(\prod_{i=1}^{k-1} \sup\limits_{\theta'\in \mathcal{E}_{\theta}} \max_{s\in S} \frac{f_{\theta',i}(s,Z_i)}{f_{\theta'}(s,Z_i)}<\infty\right) &= 1,\\
\E_{\theta^*}^{\pi}\left[ 
\sup\limits_{\theta'\in \mathcal{E}_{\theta}} \max\limits_{s'\in S}\frac{f_{\theta',n}(s',Z_n)}{f_{\theta'}(s',Z_n)}\mid X_n =s \right] 
&< \infty,\quad \forall s\in S, n\geq k, \notag
\end{align} 
and
\begin{align}
\label{al: theta expect}
  \lim\limits_{n\rightarrow\infty} \left(\E_{\theta^*}^{\pi}\left[ \sup\limits_{\theta'\in \mathcal{E}_{\theta}}
  \max\limits_{s'\in S}  \frac{f_{\theta',n}(s',Z_n)}{f_{\theta'}(s',Z_n)}\mid X_n=s \right] \right)= 1,\quad \forall s\in S.
\end{align}

\end{enumerate}
\begin{remark}
\ref{en: C1} guarantees in particular  that $m(Z_n,Y_n)$ 
converges $\P_{\theta*}$-a.s. to zero
whereas  \ref{en: C2} ensures that the ratio of 
$p^{\nu}_{\theta^*}(z_1,\ldots,z_n)$ and 
$q^{\nu}_{\theta^*}(z_1,\ldots,z_n)$ does not diverge exponentially or faster.
Assumption \ref{en: C3} is needed to carry over the consistency
of the QMLE to the MLE. In particular it implies that 
for all $\theta\not\sim\theta^*$ the ratio of 
$p^{\nu}_{\theta}(z_1,\ldots,z_n)$ and 
$q^{\nu}_{\theta}(z_1,\ldots,z_n)$ does not diverge exponentially or faster uniformly
in $\mathcal{E}_{\theta}$.
\end{remark}

\subsubsection*{Well behaving HMM}
\noindent
It is plausible that we are only able to prove consistency in the case 
where the unobservable sequence
 $Y$ would lead to a consistent estimator of $\theta^*$, itself.
To guarantee that this is indeed the case we assume:
\begin{enumerate}[label=(H\arabic*)]
\item\label{en: H1} For all $s\in S$ let 
$\E^{\pi}_{\theta^*}\left[\abs{\log f_{\theta^*}(s,Y_1)}\right] < \infty$.
\item\label{en: H2} For every $\theta\in \Theta$ with $\theta \not \sim \theta^*$, 
there exists a neighborhood $\mathcal{U}_{\theta}$
				 of $\theta$ such that
\begin{equation*}
\E^{\pi}_{\theta^*}\left[\sup\limits_{\theta'\in\mathcal{U}_{\theta}}(\log f_{\theta'}(s,Y_1))^+\right] < \infty\qquad \text{ for all } s\in S.
\end{equation*}
\item\label{en: H3} 
The mappings $\theta\mapsto f_{\theta}(s,y)$ and $\theta\mapsto f_{\theta,n}(s,y)  $
				are continuous for any 
				$s\in S$, $n\in\mathbb{N}$ and $y\in G$.
\item \label{en: H4} 
For all $s\in S$ and $n\in \N$ let
$\E^{\pi}_{\theta^*}\left[\abs{\log f_{\theta^*,n}(s,Z_n)}\right] < \infty$.
\end{enumerate}
\begin{remark}
The conditions \ref{en: H1}--\ref{en: H3} coincide with the assumptions 
in \cite[Sect. 3.2.]{douc_consistency_2011} for finite state models and
guarantee that the MLE for $\theta^*$ based on observations of $Y$ is consistent.
The condition~\ref{en: H4} is an additional regularity assumption required for the inhomogeneous
setting.
\end{remark}

 \subsection{Consistency theorem}
\noindent
Now we formulate our main results about the consistency of the QMLE and the 
MLE. 
\begin{thm}\label{thm: main_thm}
Assume that the irreducibility and continuity conditions \ref{en: P1}, \ref{en: P2},
the proximity conditions
\ref{en: C1}, \ref{en: C2} and
the well behaving HMM conditions
\ref{en: H1}--\ref{en: H4} are satisfied. Further, let  
the initial distribution
$\nu\in\mathcal{P}(S)$ be strictly positive if and only if $\pi$ is strictly positive. 
Then 
\[
\theta^{\, \rm QML}_{\nu,n}\rightarrow\theta^*,\quad \P_{\theta^*}^{\pi}\text{-a.s.}
\]
as $n\rightarrow\infty$.
\end{thm}
Note that condition \ref{en: C3} is not required in the previous statement.
We only need it to prove the consistency of the MLE $\theta^{\rm\,ML}_{\nu,n}$.
 
\begin{coro} \label{lemma: main lemma} 
Assume that the setting and conditions of 
Theorem \ref{thm: main_thm} and \ref{en: C3} are satisfied.
Then
\[
\theta^{\rm\,ML}_{\nu,n}\rightarrow\theta^*,\quad \P_{\theta^*}^{\pi}\text{-a.s.}
\]
as $n\rightarrow\infty$.
\end{coro}

 \section{Application}
 \label{sec: appl}
\noindent
 We consider two models where we explore the
 structural assumptions from Section~\ref{sec: struct_cond} explicitly.
 The Poisson model, see Section~\ref{Poisson models with Poisson errors},
 illustrates a simple example with countable observation space.
 The linear Gaussian model is an extension of the model 
introduced in \eqref{dyn} and \eqref{eq: Y_ion} to 
multivariate and possibly correlated observations.

\subsection{Poisson DHMM}\label{Poisson models with Poisson errors}
\noindent
For $i=1,\dots,K$ let $\lambda_{\theta^*}^{(i)}>0$ and define the vector
$\lambda_{\theta^*}=(\lambda_{\theta^*}^{(1)},\ldots,\lambda_{\theta^*}^{(K)})$.
Conditioned on $X$ the non-observed homogeneous sequence $Y=(Y_n)_{n\in\N}$ is 
an independent sequence of Poisson-distributed random variables with parameter $\lambda_{\theta^*}^{(X_n)}$. 
In other words, given $X_n=x_n$ we have $Y_n\sim \Poi(\lambda_{\theta^*}^{(x_n)})$. 
Here $\Poi(\alpha)$ denotes the Poisson distribution with expectation $\alpha>0$.
The observed sequence $Z=(Z_n)_{n\in\N}$ is determined by
\[
Z_n = Y_n + \varepsilon_n,
\]
where $(\varepsilon_n)_{n\in\N}$ is an independent sequence of random variables with 
$\epsilon_n \sim \Poi(\beta_n)$. 
Here $(\beta_n)_{n\in\N}$ is a sequence of positive real numbers 
satisfying for some $p>1$ that
\begin{equation}\label{Poisson beta}
\beta_n = \mathcal{O}(n^{-p}).
\end{equation}
We also assume that $(\varepsilon_n)_{n\in\N}$ is independent of 
$Y$ and that the parameter
$\theta$ determines the transition matrix $P_{\theta}$ and the intensity $\lambda_{\theta}$ continuously. 
Note that 
the observation space is given by $G=\N\cup\{0\}$ equipped with 
the counting measure $\lambda$. 
Figures \ref{Pois example} illustrates the empirical mean square error
of approximations of the MLEs.
\begin{figure}[!htb]
	\centering
  \includegraphics[width=0.9\textwidth]{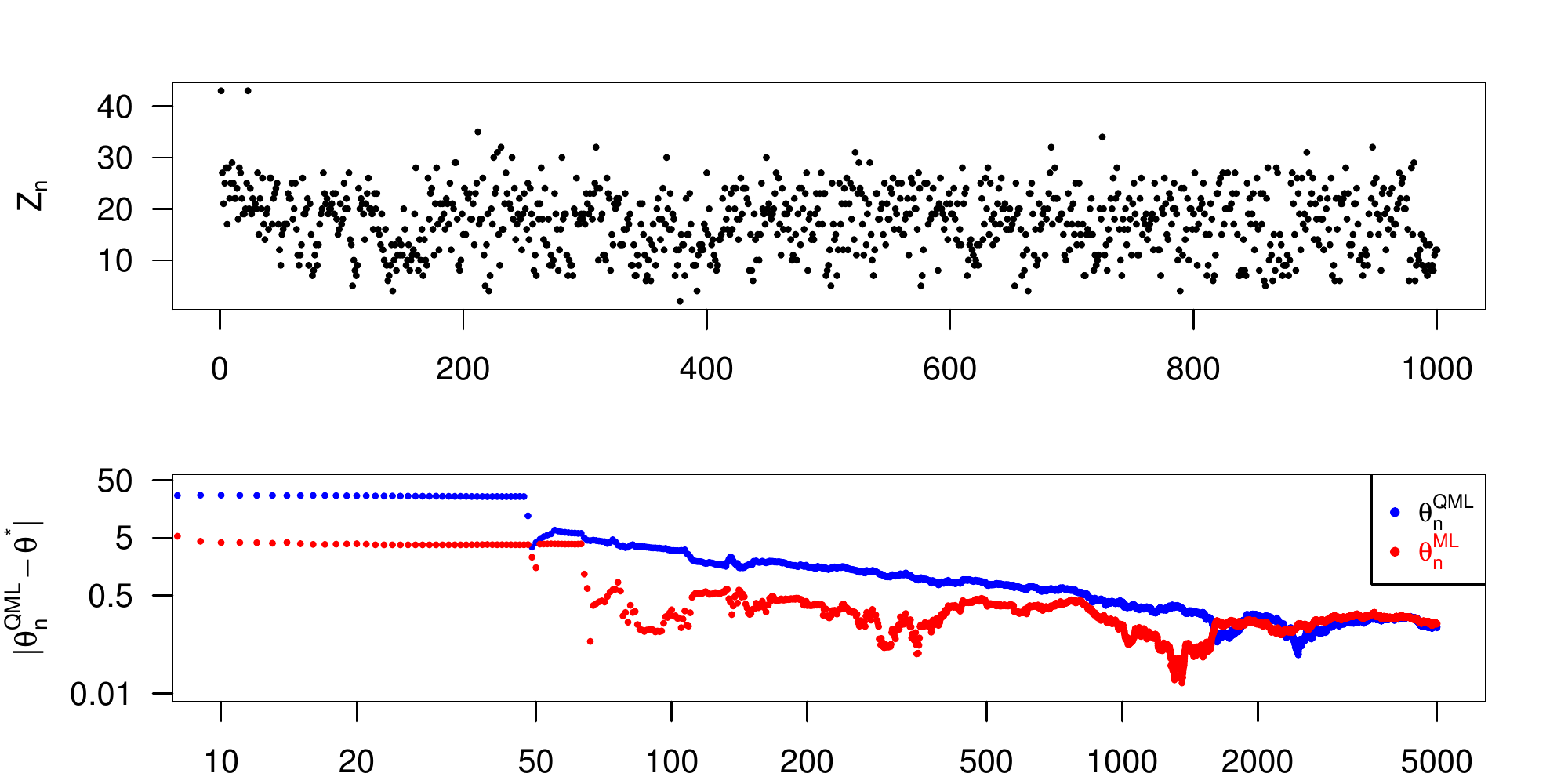}
	\caption{
	Exemplary trajectory of observations of the Poisson model from 
	Section~\ref{Poisson models with Poisson errors} (above) 
	and Euclidean norm of the difference of $\theta^*$ and 
	the estimators  based on a single trajectory of observations (below).  Here 
	$K=2$, $n=5\cdot10^3$ and $\theta^* = (10,20,0.8,0.1)$. The parameter $\theta^*$ determines 
	$\lambda_{\theta^*} = (10,20)$ 
	and the ``true'' transition matrix by $P_{\theta^*}(1,1) = 0.8$, $P_{\theta^*}(2,1) = 0.1$.
	The inhomogeneous noise is driven by an intensity $\beta_n = 40\, n^{-1.01}$.
	}
	\label{Pois example}
\end{figure}

\begin{figure}[!htb]
	\centering
  \includegraphics[width=0.8\textwidth]{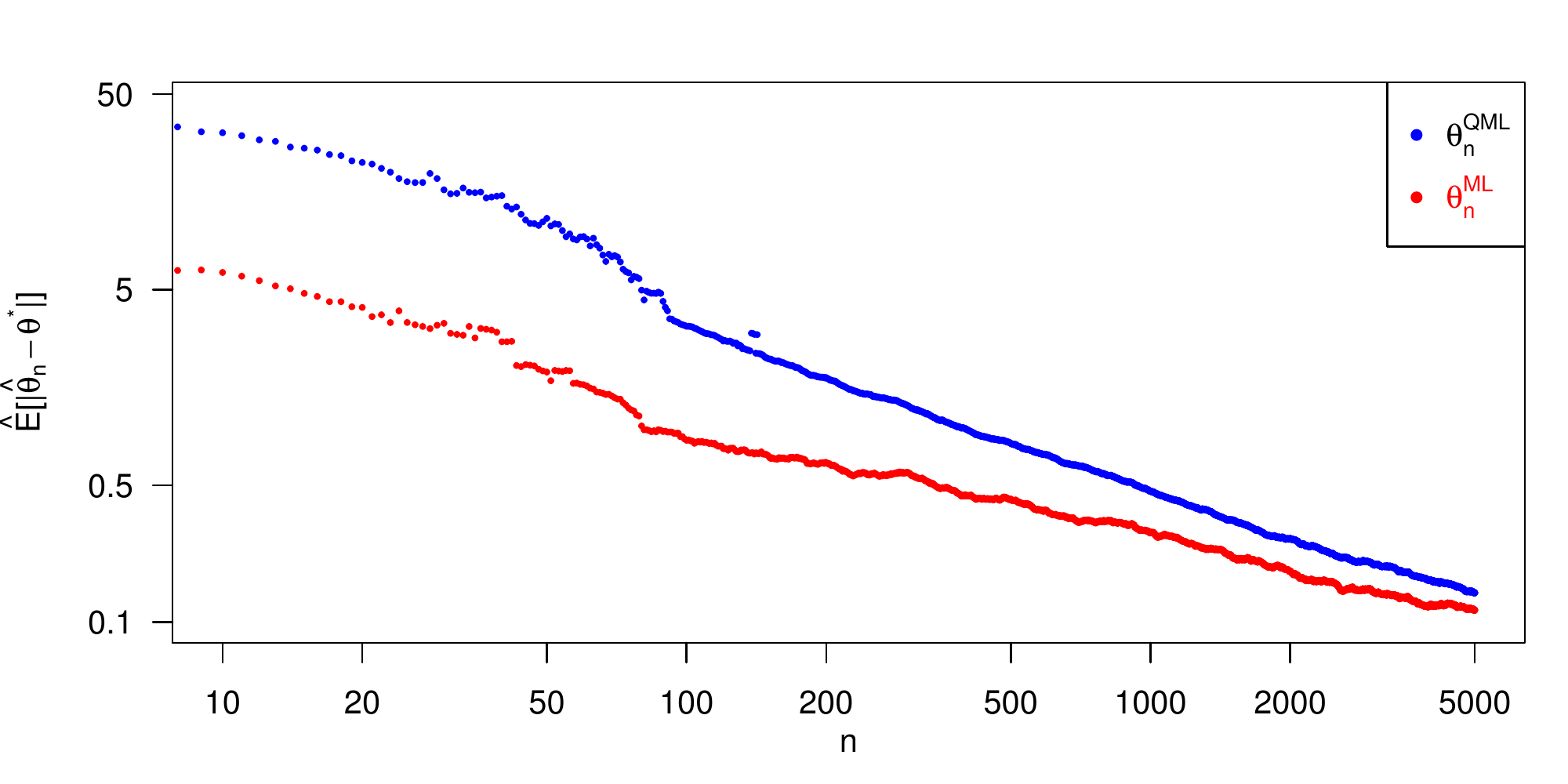}
	\caption{
	Empirical mean of the Euclidean norm of the difference of $\theta^*$ and the estimators
	based on $100$ i.i.d. replications of the DHMM.
	Here 
	$K=2$, $n=5\cdot10^3$ and $\theta^* = (10,20,0.8,0.1)$. The parameter $\theta^*$ determines 
	$\lambda_{\theta^*} = (10,20)$ 
	and the ``true'' transition matrix by $P_{\theta^*}(1,1) = 0.8$, $P_{\theta^*}(2,1) = 0.1$.
	The inhomogeneous noise is driven by an intensity $\beta_n = 40\, n^{-1.01}$.
	}
	\label{Pois example mean}
\end{figure}

To obtain the desired consistency of the two estimators we need to check
the conditions \ref{en: P1}, \ref{en: P2}, \ref{en: C1}--\ref{en: C3} 
and \ref{en: H1}--\ref{en: H4}:\\[-1ex]

 \noindent
\emph{To \ref{en: P1} and \ref{en: P2}:} By the assumptions in this scenario
those conditions are satisfied.\\[-1ex]

\noindent
\emph{To \ref{en: H1}--\ref{en: H4}:} 
For $\theta\in\Theta$, $s\in S$ and $y\in G$ we have
\begin{align*}
\abs{\log f_{\theta}(s,y)} = -\log \left(\frac{\left(\lambda_{\theta}^{(s)}\right)^{y}}{y!} \exp(-\lambda_{\theta}^{(s)})\right)
&= - y \log (\lambda_{\theta}^{(s)}) + \log(y!) + \lambda_{\theta}^{(s)}\\
&\leq -y \log(\lambda_{\theta}^{(s)}) + y^2 + \lambda_{\theta}^{(s)}.
\end{align*}
Hence
\begin{align*}
&\E_{\theta^*}^{\pi}\left[\abs{\log f_{\theta^*}(s,Y_1)}\right] \\
& \leq  -\log(\lambda_{\theta^*}^{(s)}) \sum_{s=1}^K \pi (s) \lambda_{\theta^*}^{(s)}
  +
 \sum_{s=1}^K \pi (s) \left(\left(\lambda_{\theta^*}^{(s)}\right)^2 + 
 \lambda_{\theta^*}^{(s)}  \right) + \lambda_{\theta^*}^{(s)}<\infty
\end{align*}
and \ref{en: H1} is verified. A similar calculation gives \ref{en: H4}.
Condition \ref{en: H2} follows simply by $(\log f_{\theta}(s,y))^+=0$.
Condition \ref{en: H3} follows by the continuity in 
the parameter of the probability function of
the Poisson distribution and the continuity 
of the mapping $\theta\mapsto (P_\theta,\lambda_\theta)$.
\\[-1ex]

\noindent
\emph{To \ref{en: C1} -- \ref{en: C3}:} 
For any $\delta>0$
and any $s\in S$ we have
\begin{align*}
\P_{\theta}^{\pi}\left(\abs{Z_n-Y_n}\geq \delta\mid X_n = s\right) 
= \P_{\theta}^{\pi}\left(\abs{\varepsilon_n}\geq \delta\right)
\leq 1 - \P_{\theta}^{\pi}\left(\varepsilon_n =0 \right)
= 1 -\exp(-\beta_n).
\end{align*}
From \eqref{Poisson beta} it follows that
\[
1-\exp(-\beta_n) = \mathcal{O}(n^{-p}),
\]
which proves \ref{en: C1}. Observe that for any $s\in S,z\in G$ we have
\begin{align*}
\max\limits_{s\in S}\frac{f_{\theta^*,n}(s,z)}{f_{\theta^*}(s,z)}
=\max\limits_{s\in S} \frac{\left(\beta_n+ \lambda_{\theta^*}^{(s)}\right)^{z} }{\left( \lambda_{\theta^*}^{(s)}\right)^{z}}\exp(-\beta_n)
= \left( a_n\right)^{z}\exp(-\beta_n),
\end{align*}
with $a_n =  \max\limits_{s\in S} \frac{\beta_n+ \lambda_{\theta^*}^{(s)}}{\lambda_{\theta^*}^{(s)}}$. 
Now we verify \ref{en: C2} with $k=1$. 
For all $n\in\N$ and $s\in S$ we have
\begin{align*}
\E_{\theta^*}^{\pi}\left[\max\limits_{s'\in S}\frac{f_{\theta^*,n}(s',Z_n)}{f_{\theta^*}(s',Z_n)}\mid X_n = s\right]
&= \E_{\theta^*}^{\pi}\left[a_n^{Z_n} \exp(-\beta_n)\mid X_n = s\right] \\
&= \exp\left((\lambda_{\theta^*}^{(s)}+\beta_n)(a_n-1)-\beta_n\right)<\infty.
\end{align*}
Fix $s\in S$, and note that
\begin{align*}
&\limsup\limits_{n\rightarrow\infty} 
\E_{\theta^*}^{\pi}\left[  
\max\limits_{s'\in S}\frac{f_{\theta^*,n}(s',Z_n)}{f_{\theta^*}(s',Z_n)} \mid X_n=s \right] \\
& = \limsup\limits_{n\rightarrow\infty} \exp\left((\lambda_{\theta^*}^{(s)}+\beta_n)(a_n-1)-\beta_n\right)=1.
\end{align*}

The last equality follows by the fact that $\lim_{n\to \infty} a_n =1$ and $\lim_{n\to \infty}\beta_n = 0$.
Condition \ref{en: C3} follows by similar arguments.

The application of Theorem~\ref{thm: main_thm} and Corollary~\ref{lemma: main lemma}
leads to the following result.

\begin{coro}
 For any initial distribution $\nu\in\mathcal{P}(S)$ which is strictly positive if 
 and only if $\pi$ is strictly positive, we have
 for the Poisson DHMM if \eqref{Poisson beta} holds for some $p>1$ that
\[
\theta^{\,\rm QML}_{\nu,n}\rightarrow \theta^*, \quad \P_{\theta^*}^{\pi}\text{-a.s.}
\]
and 
\[
\theta^{\,\rm ML}_{\nu,n}\rightarrow \theta^*, \quad \P_{\theta^*}^{\pi}\text{-a.s.}
\]
as $n\rightarrow\infty$.
\end{coro}

\subsection{Multivariate linear Gaussian DHMM}\label{linear Gaussian models}
\noindent
For $i=1,\dots,K$ let $\mu_{\theta^*}^{(i)}\in \R^M$, 
$\Sigma_{\theta^*}^{(i)}\in \R^{M\times M}$ with full rank, where $M\in\N$. Define
$\mu_{\theta^*}=(\mu_{\theta^*}^{(1)},\ldots,\mu_{\theta^*}^{(K)})$
as well as $\Sigma_{\theta^*} = (\Sigma_{\theta^*}^{(1)},\ldots,\Sigma_{\theta^*}^{(K)}).$
 The sequences $Y=(Y_n)_{n\in\N}$
and $Z=(Z_n)_{n\in\N}$ are defined by
\begin{align*}
    Y_n & = \mu_{\theta^*}^{(X_n)} + \Sigma_{\theta^*}^{(X_n)} V_n \\
    Z_n & = Y_n + \varepsilon_n.
\end{align*}
Here $(V_n)_{n\in\N}$ is an i.i.d. sequence of random vectors with 
$V_n\sim\normal(0,I)$, where $I\in \R^{M\times M}$ denotes the identity matrix, 
and $(\varepsilon_n)_{n\in\N}$ is a sequence of independent random vectors with
$\varepsilon_n\sim  \normal(0,\beta_n^2 I)$, where $(\beta_n)_{n\in\N}$ is a positive real-valued sequence 
satisfying for some $q>0$
that
\begin{equation}\label{beta}
\beta_n = \mathcal{O}(n^{-q}).
\end{equation}
Here we also assume that the mapping $\theta\mapsto (P_{\theta},\mu_{\theta},\Sigma_{\theta})$ 
is continuous.
Furthermore, note that $G=\R^M$ 
and $\lambda$ is the $M$-dimensional Lebesgue measure. 
Figures \ref{LGM example} illustrates the empirical mean square error of approximations 
of the MLEs.

\begin{figure}[!htb]
	\centering
  \includegraphics[width=0.9\textwidth]{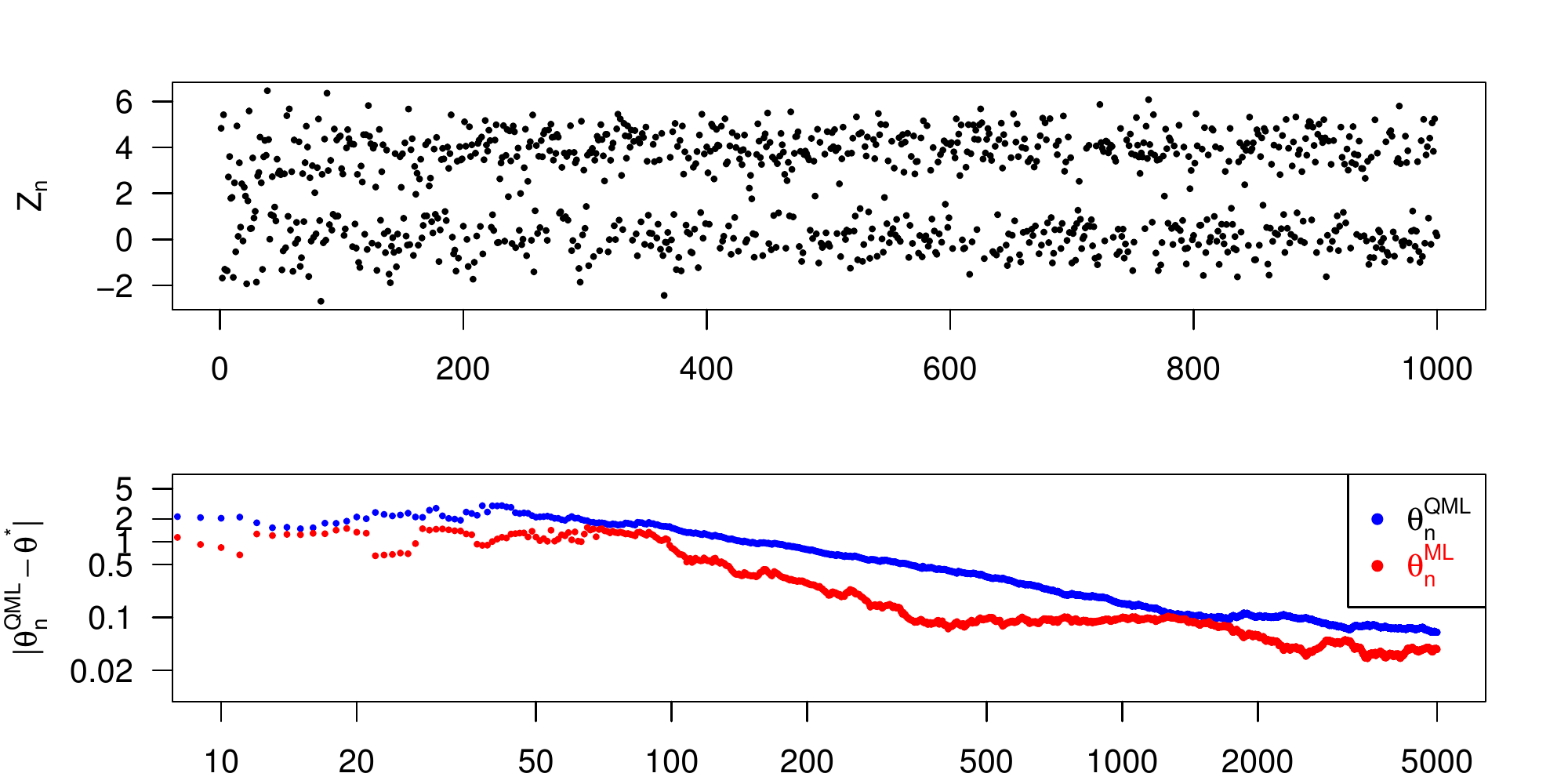}
	\caption{
	Exemplary trajectory of observations of the  linear Gaussian model from 
	Section~\ref{linear Gaussian models} (above) 
	and Euclidean norm of the difference of $\theta^*$ and the estimators 
	 based on a single trajectory of observations (below).
	Here 
	$M=1$, $K=2$, $n= 5\cdot 10^3$ and $\theta^* = (0,4,0.5,0.5,0.4,0.5)$.
	The parameter $\theta^*$ determines
	$\mu_{\theta^*} = (0,4)$, $\Sigma_{\theta^*} = (0.5,0.5)$ and the ``true'' transition matrix by 
	$P_{\theta^*}(1,1) = 0.4$, $P_{\theta^*}(2,1) = 0.5$. 
	The inhomogeneous noise is driven by an intensity $\beta_n = 10\, n^{-0.75}$.}
	\label{LGM example}
\end{figure}

\begin{figure}[!htb]
	\centering
  \includegraphics[width=0.8\textwidth]{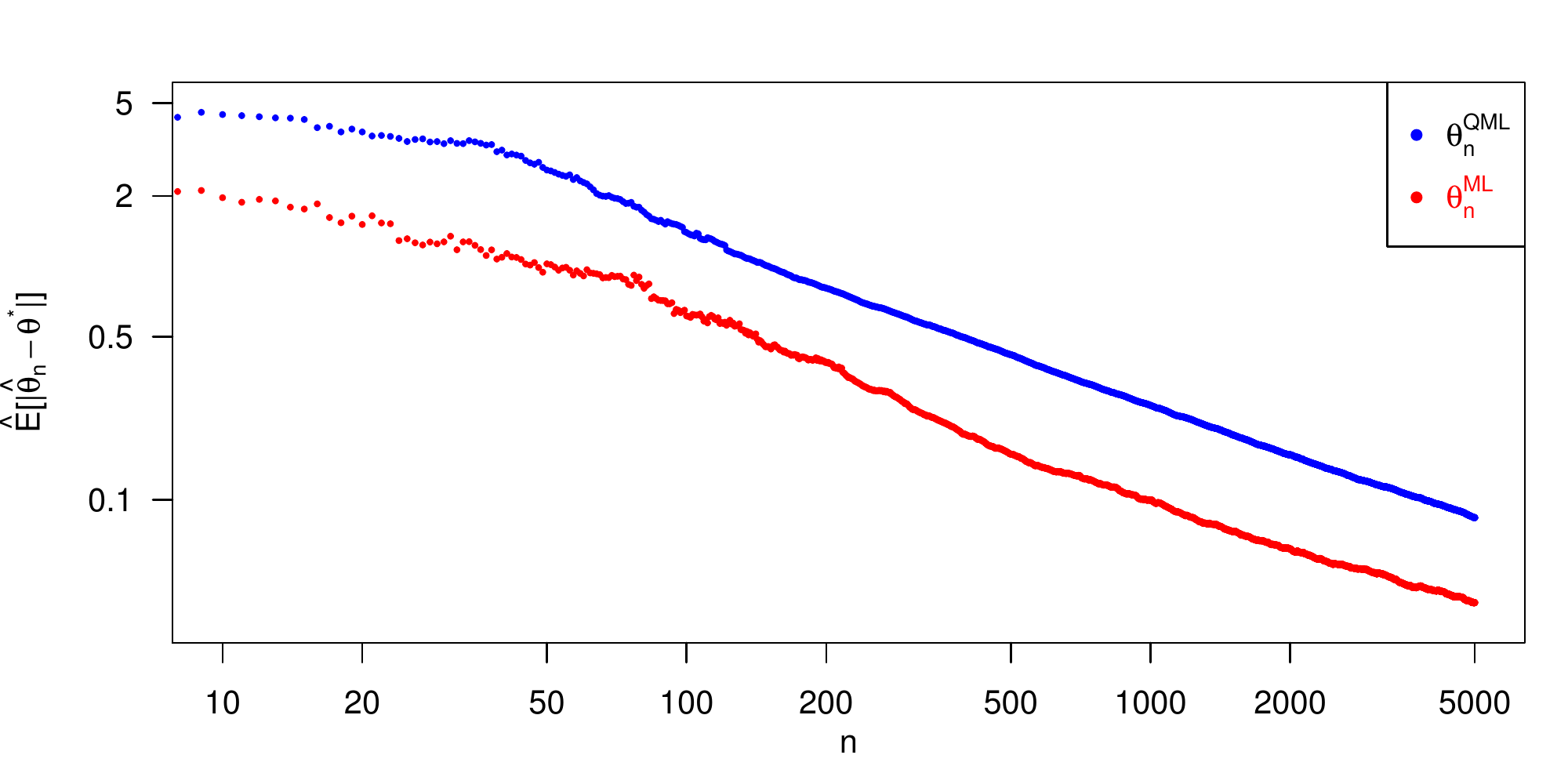}
	\caption{
	Empirical mean of the Euclidean norm of the difference of $\theta^*$ 
	and the estimators norm based on $100$ i.i.d. replications of the DHMM.
	Here 
	$M=1$, $K=2$, $n= 5\cdot 10^3$ and $\theta^* = (0,4,0.5,0.5,0.4,0.5)$.
	The parameter $\theta^*$ determines
	$\mu_{\theta^*} = (0,4)$, $\Sigma_{\theta^*} = (0.5,0.5)$ and the ``true'' transition matrix by 
	$P_{\theta^*}(1,1) = 0.4$, $P_{\theta^*}(2,1) = 0.5$. 
	The inhomogeneous noise is driven by an intensity $\beta_n = 10\, n^{-0.75}$.}
	\label{GLM example mean}
\end{figure}

To obtain consistency of the two estimators we need to check
the conditions \ref{en: P1}, \ref{en: P2}, \ref{en: C1}--\ref{en: C3} 
and \ref{en: H1}--\ref{en: H4}:\\[-1ex]

\noindent
\emph{To \ref{en: P1} and \ref{en: P2}:} 
By definition of the model this conditions are satisfied. \\[-1ex]

\noindent
\emph{To \ref{en: H1}--\ref{en: H4}:} 
For a matrix $A\in \R^{M\times M}$ denote $A^2 = AA^T$ and $A^{-2}=(A^2)^{-1}$. 
Note that for $s\in S$, $\theta\in \Theta$ and $y,z\in G$ we have
by
\begin{align*}
f_{\theta}(s,y) 
&= \frac{(2\pi)^{-M/2}}{\det\left(\left(\Sigma_{\theta}^{(s)}\right)^2\right)^{1/2}} 
\exp\left(-\frac{1}{2} (y - \mu_{\theta}^{(s)})^T \left(\Sigma_{\theta}^{(s)}\right)^{-2}(y - \mu_{\theta}^{(s)})\right),\\
f_{\theta,n}(s,z) &= \frac{(2\pi)^{-M/2}}{ \det\left(\left(\Sigma_{\theta}^{(s)}\right)^2
+ \beta_n^2 I\right)^{1/2}} \\ 
& \qquad \qquad \qquad \times \exp\left(-\frac{1}{2} (z - \mu_{\theta}^{(s)})^T 
\left(\left(\Sigma_{\theta}^{(s)}\right)^2+ \beta_n^2 I\right)^{-1}(z - \mu_{\theta}^{(s)})\right).
\end{align*} 
Further, observe that $\det\left(\left(\Sigma_{\theta}^{(s)}\right)^2\right)>0$  
for all $s\in S$. For some constant $C_1>0$ 
we have
\begin{align*}
\E_{\theta^*}^{\pi}\left[\abs{\log f_{\theta}(s,Y_1) }\right] 
&\leq C_1 + \E_{\theta^*}^{\pi}\left[ \frac{1}{2} (Y_1 -  \mu_{\theta}^{(s)})^T 
\left(\Sigma_{\theta}^{(s)}\right)^{-2}(Y_1 -  \mu_{\theta}^{(s)})\right] 
< \infty,
\end{align*}
since for each $i,j\in \{1,\dots,M\}$ 
we have $\E_{\theta^*}^{\pi}\left[Y_1^{(i)} Y_1^{(j)} \right]<\infty$
with the notation \mbox{$Y_1=(Y_1^{(1)},\dots,Y_1^{(M)})$}. 
By this estimate \ref{en: H1}
and \ref{en: H2} follows easily.
Condition \ref{en: H4} follows by similar arguments. More detailed, we have that $\beta_n^2$ is 
finite and converges to zero as well, as that 
there exists a constant $C_2>0$ such that 
\begin{align*}
\E_{\theta^*}^{\pi}\left[\abs{\log f_{\theta^*,n}(s,Z_n) }\right] 
&\leq C_2 + \E_{\theta^*}^{\pi}\left[ \frac{1}{2} (Z_n - \mu_s)^T 
\left(\left(\Sigma_{\theta}^{(s)}\right)^2 + \beta_n^2 I\right)^{-1}(Z_n - \mu_s)\right]. 
\end{align*}
For all $n\in\N$ the right-hand side of the previous inequality is finite, since 
for each $i,j\in\{1,\dots,M\}$ we have $\E_{\theta^*}^{\pi}\left[Z_n^{(i)} Z_n^{(j)} \right]<\infty$,
with $Z_n=(Z_n^{(1)},\dots,Z_n^{(M)})$.
Finally condition \ref{en: H3} is satisfied by the continuity of the conditional density and the continuity
of the mapping $\theta\mapsto (P_\theta,\mu_\theta,\Sigma_\theta)$.\\[-1ex]

\noindent
\emph{To \ref{en: C1} -- \ref{en: C3}:}  
Here $m$ is the Euclidean metric in $\R^M$ such that 
$\abs{\varepsilon_n} = m(Y_n,Z_n)$.
Fix some $r>0$ with $r/q>1$  and observe that
for any $\delta>0$ and $s\in S$ we have
\begin{align*}
\P_{\theta^*}^{\pi}\left( m(Y_n,Z_n)>\delta\mid X_n=s  \right)
=\P_{\theta^*}^{\pi}\left( \abs{\varepsilon_n}>\delta  \right)
=\P_{\theta^*}^{\pi}\left(\beta_n^r \abs{N}^r>\delta^r  \right)
\leq \frac{\E_{\theta^*}^{\pi}\left[ \abs{N}^r  \right]\beta_n^r}{\delta^{r}},
\end{align*}
where $N\sim\normal(0,I)$. 
By the fact that $\E_{\theta^*}^{\pi}\left[ \abs{N}^r  \right]<\infty$
and \eqref{beta} we obtain that condition \ref{en: C1} is satisfied with $p=r/q>1$. 

The requirement of \eqref{eq: finite} of \ref{en: C2} holds for any $k\in\N$,
since the density of normally distributed random vectors is strictly positive and finite. 
Observe that
\begin{align*}
&\max\limits_{s\in S} \frac{f_{\theta,n}(s,Z_n)}{f_{\theta}(s,Z_n)}\leq \\
&\qquad C_n \max\limits_{s\in S}
\exp\left( -\frac{1}{2}  (Z_n-\mu_{\theta}^{(s)})^T\left((
(\Sigma^{(s)}_\theta)^2+\beta_n^2 I)^{-1} -(\Sigma^{(s)}_\theta)^{-2}\right)(Z_n-\mu_{\theta}^{(s)})  \right),
\end{align*}
with
\[
C_n := 
\max\limits_{s\in S}\frac{\left(\det\left(\left(\Sigma_{\theta}^{(s)}\right)^2\right)\right)^{1/2}}
{\left(\det\left(\left(\Sigma_{\theta}^{(s)}\right)^2+\beta_n^2I\right) \right)^{1/2}}.
\]
Note that $\lim_{n\rightarrow\infty}C_n =1$. 
Since for an invertible matrix $A\in\R^{M\times M}$, $A\mapsto A^{-1}$ 
is continuous and $\Sigma_{\theta^*}^{s}$ has full rank, it follows that
\[
\lim\limits_{n\rightarrow\infty} 
\left( \left(\Sigma_{\theta}^{(s)}\right)^2+\beta_n^2I \right)^{-1} 
=\left( \Sigma_{\theta}^{(s)}\right)^{-2}.
\]f
Set $(\Sigma_{\theta}^{(s)})^2_n := (\Sigma_{\theta}^{(s)})^2 + \beta_n^2 I$
and define  $B_n = B_{n,s} :=  (\Sigma_{\theta}^{(s)})^{-2} - (\Sigma_{\theta}^{(s)})_n^{-2} $. 
Note that the entries of $B_n$ converge to zero when $n$
goes to infinity.
 
Further, by the fact that $(B_n)_{n\in\N}$ 
is a sequence of symmetric, positive definite matrices there exist sequences of  
orthogonal matrices $(U_n)_{n\in\N}\subset \R^{M\times M}$ 
and diagonal matrices $(D_n)_{n\in N}\subset \R^{M\times M}$ such that
\[
B_n = U_n^T D_n^{1/2} D_n^{1/2} U_n.
\] 
Of course, $U_n$ and $D_n$ depend on $s$.
We define a sequence of random vectors $(W_{n,s})_{n\in \N}$ by 
setting $W_{n,s} := U_n D_n^{1/2}(Z_n-\mu_{\theta}^{(s)})$,
such that 
\begin{align*}
&(Z_n-\mu_{\theta}^{(s)})^T\left((\Sigma^{(s)}_\theta)^{-2}-((\Sigma^{(s)}_\theta)^2+\beta_n^2 I)^{-1}\right)(Z_n-\mu_{\theta}^{(s)})  \\
&=(Z_n-\mu_{\theta}^{(s)})^T B_n(Z_n-\mu_{\theta}^{(s)})  
 =  W_{n,s}^T  W_{n,s}.
\end{align*}
The random variable $Z_i$ conditioned on $X_i=x$ is normally distributed 
with mean $\mu_{\theta}^{(x)}$
and covariance matrix $(\Sigma_{\theta}^{(x)})^2_n$.
Hence $W_{i,s}$, conditioned on $X_i=x$, satisfies
\[
W_{i,s}\sim \normal(\tilde{\mu}_i ,A_i),
\]
with 
\[
\tilde{\mu}_i = U_i^T D_i^{1/2}(\mu_{\theta}^{(x)}-\mu_{\theta}^{(s)})
\]
and
\[ 
A_i = U_i^T D_i^{1/2}(\Sigma_{\theta}^{(x)})^2_i (U_i^T D_i^{1/2})^T.
\]
Since $A_i$ is symmetric and positive definite, we find sequences 
of orthogonal matrices $(U'_n)_{n\in\N}$ and diagonal 
matrices $(D_n')_{n\in\N}$ depending on $x$ and $s$ such that
\[
A_i = U_i'D_i'^{1/2}D_i'^{1/2} U_i'^T.
\]
 Let $(N_i)_{i\in\N}$ be an i.i.d. sequence of random vectors with $N_i\sim \mathcal{N}(0,I)$
 and denote $N_i=(N_i^{(1)},\dots,N_i^{(M)})$. 
 Then
\begin{align*}
W_{i,s}^TW_{i,s}  &= \abs{W_{i,s}}^2
\stackrel{\mathcal{D}}{=} \abs{U_i'D_i'^{1/2}(N_i + D_i'^{-1/2} U_i'^T\tilde{\mu}_i )}^2\\
&= \abs{D_i'^{1/2}(N_i + D_i'^{-1/2} U_i'^T\tilde{\mu}_i)}^2
= \sum_{j=1}^M D'_i(j,j) \left(N_i^{(j)} + ( D_i'^{-1/2} U_i'^T\tilde{\mu}_i)^{(j)} \right)^2.
\end{align*}
Recall 
that for a chi-squared distribution 
with one degree of freedom and non-centrality 
parameter $\gamma>0$ the moment generating function in $t$, with $t<1/2$, is given by 
$
\frac{\exp(\gamma t/(1-2t))}{(1-2t)^{1/2}}$.
Hence, for any $t<\min_{j=1,\dots,M} D'_i(j,j)^{-1}$ 
with
non-centrality 
parameter $( D_i'^{-1/2} U_i'^T\tilde{\mu}_i)^{(j)}$ at $\frac{t}{2} D_i'(j,j)$ it  is 
well-defined
and 
we obtain
\begin{align*}
&\E_{\theta^*}^{\pi}\left[ \exp\left(\frac{t}{2} W_{i,s'}^TW_{i,s'}\right)\mid X_i = s\right]\\
&= \prod\limits_{j=1}^M (1-2(\frac{t}{2})D'_i(j,j))^{-1/2}
 \exp\left(\frac{( D_i'^{-1/2} U_i'^T\tilde{\mu}_i)^{(j)}(\frac{t}{2})D'_i(j,j) }
{1-2 (\frac{t}{2})D'_i(j,j) }\right) \\
&= \prod\limits_{j=1}^M (1-tD'_i(j,j))^{-1/2}
 \exp\left(\frac{( D_i'^{-1/2} U_i'^T\tilde{\mu}_i)^{(j)}(\frac{t}{2})D'_i(j,j) }
{1-tD'_i(j,j) }\right) \rightarrow 1
\end{align*}
as $i\rightarrow\infty$, since $\lim\limits_{i\rightarrow\infty} D_i'(j,j) = 0$ 
for all $j=1,\dots,M$.
We can choose $k$ sufficiently large, such that
$K<\min_{j=1,\dots,M} D'_i(j,j)^{-1}$ for all $i\geq k$. We find that
\begin{align*}
\E_{\theta^*}^{\pi}& \left[ \max\limits_{s'\in S}  \exp\left( \frac{1}{2} W_{k,s'}^TW_{k,s'}\right)\mid X_k = s\right]
 \leq \E_{\theta^*}^{\pi}\left[  \sum\limits_{s'\in S}\exp\left(\frac{1}{2} W_{k,s'}^TW_{k,s'}\right)\mid X_k=s\right] \\
& \leq \prod  \limits_{s'\in S}   \left(\E_{\theta^*}^{\pi} \left[ \exp\left(\frac{K}{2} W_{k,s'}^TW_{k,s'}\right)\mid X_k=s\right]\right)^{1/K} ,
\end{align*}
where we used the generalized H\"older inequality in the last estimate. 
Then, by
taking the limit superior we obtain that 
the right-hand side of the previous inequality goes to one 
for $k\to \infty$ such that \ref{en: C2} holds. 
Condition \ref{en: C3} can be verified similarly.

The application of Theorem~\ref{thm: main_thm} and Corollary~\ref{lemma: main lemma}
leads to the following result.

\begin{coro}
 For any initial distribution $\nu\in\mathcal{P}(S)$ which is strictly positive if 
 and only if $\pi$ is strictly positive, we have
 for the multivariate Gaussian DHMM satisfying \eqref{beta} for some $q>0$ that
\[
\theta^{\,\rm QML}_{\nu,n}\rightarrow \theta^*, \quad \P_{\theta^*}^{\pi}\text{-a.s.}
\]
and 
\[
\theta^{\,\rm ML}_{\nu,n}\rightarrow \theta^*, \quad \P_{\theta^*}^{\pi}\text{-a.s.}
\]
as $n\rightarrow\infty$.
\end{coro}

\begin{remark}
 For $K=2$ and $M=1$ we have the model of the conductance level of ion channel data with varying voltage 
 provided in the introduction, see Figure \ref{example ramp} and \eqref{eq: Y_ion} and \eqref{eq: Z_ion}. 
 The previous corollary states the desired
 consistency of the considered MLEs in that setting. 
 A data analysis of the ion channel recordings of the underlying DHMM will be done
 in a separate paper.
\end{remark}

\section{Discussion and limitations}
\label{sec: disc}
\noindent
In this section we discuss four aspects. First, having the models from Section~\ref{sec: appl} in mind,
one might consider a hybrid case, that is, e.g. if the non-observed sequence $Y$ is
Poisson distributed and the inhomogeneous noise is normally distributed. We discuss where our approach fails here and
provide a strategy how to resolve this issue. Second, one might ask whether the proximity assumptions formulated in 
Section~\ref{sec: struct_cond} can be relaxed. We provide a simple example where \ref{en: C1} is not satisfied
and $\theta_{\nu,n}^{\rm QML}$ is not consistent anymore.
Third, we discuss the restriction of considering only hidden Markov chains on finite state spaces. Finally, we comment and discuss conditions which lead to asymptotic normality of the QMLE.

\subsection{Hybrid model}
\noindent

The hidden sequences $X$ and $Y$ of the DHMM are defined as in Section~\ref{Poisson models with Poisson errors}.
The observed sequence $Z=(Z_n)_{n\in\N}$ is given by
\[
Z_n = Y_n + \varepsilon_n,
\]
where $(\varepsilon_n)_{n\in\N}$ is an independent sequence of random variables with 
$\epsilon_n \sim \mathcal{N}(0,\beta^2_n)$ and 
a $(\beta_n)_{n\in\N}\subset (0,\infty)$ satisfies $\lim_{n\to \infty} \beta_n^2 =0$. 
In other words, on the Poisson random variable $Y_n$ we add Gaussian time-dependent noise.

The main issue is that the observed sequence 
$Z$ takes values in $\R$ whereas $Y$ takes values in $\N\cup \{0\}$. 
Consider $G=\mathbb{R}$ equipped with the reference measure
$\lambda(\cdot) = \mathcal{L}(\cdot) + \sum_{i=0}^{\infty} \delta_i(\cdot)$.
Here $\mathcal{L}(\cdot)$ denotes the Lebesgue measure and $\delta_{i}(\cdot)$  
the Dirac-measure at point $i\in\N$.
The conditional density $f_{\theta,n}$ w.r.t. $\lambda$ is given by
\begin{align*}
f_{\theta,n}(s,z) = 
\begin{cases}\sum_{j=0}^{\infty} \frac{\lambda_{\theta}^{(s)}}{j!} \exp(-\lambda_{\theta}^{(s)}) 
\frac{1}{(2\pi\beta_n^2)^{1/2}} \exp\left(-\frac{(z-j)^2}{2\beta_n^2}\right)&\quad z\in \R\backslash \N\\
0&\quad z\in \N.
\end{cases}
\end{align*}
One can verify that \ref{en: C2} is not satisfied in this scenario. 
In general, assumption \ref{en: C2} is difficult to handle, whenever the support
of $f_\theta$ is strictly ``smaller'' than the support of $f_{\theta,n}$. 
We mention a possible strategy to resolve this problem:
\begin{enumerate}
 \item Transform the observed sequence to a sequence 
 $\tilde{Z}=(\widetilde{Z}_n)_{n\in\mathbb{N}}$, such that the support of the
 corresponding conditional density coincides with the support of $f_\theta$.
 For example, this might be done by rounding to the nearest natural number, that is, $\widetilde{Z}_n = \lfloor Z_n+0.5 \rfloor$.
 \item Prove that the QMLE
 $\tilde{\theta}_{\nu,n}^{\rm QML}$, based on $\tilde Z$, is consistent. (For example, by applying Theorem~\ref{thm: main_thm}.)
 \item Prove that $\theta_{\nu,n}^{\rm QML}-\tilde{\theta}_{\nu,n}^{\rm QML}\to 0,$ $\mathbb{P}_{\theta^*}^\pi$ a.s. as 
 $n\to \infty$.
\end{enumerate}
A similar strategy might be used to obtain consistency for the MLE.
 
\subsection{Proximity assumption}
\noindent
We show that in general one cannot weaken the proximity assumption from Section~\ref{sec: struct_cond}.

We provide an example, which does not satisfy \ref{en: C1} and show 
that $\theta^{\rm QML}_{\nu,n}$ 
is not strongly consistent for the approximation of $\theta^*$.
\begin{example} 
Consider the linear Gaussian model of Section~\ref{linear Gaussian models}
in the case $m=1$ and $K=1$ with $\theta^* =(0,1)$.
The parameter $\theta^*$ determines the mean $\mu_{\theta^*} = 0$ and the variance $\sigma_{\theta^*}^2=1$.
Let $\beta = \lim\limits_{n\to \infty} \beta_n  >0$ and $\theta_0 = (0,1+\beta^2)$.
Note that
\[
 \epsilon_n \overset{\mathcal{D}}{\to} N,
\]
as $n\to\infty$, where $N\sim \normal(0,\beta^2)$. This contradicts the conclusion of  
Lemma \ref{Lemma a.s. convergence} below and therefore 
assumption \ref{en: C1} is not satisfied. Further we have
\[
\lim\limits_{n\to \infty}\E_{\theta^*}^{\pi}\left[\log f_{\theta}(1,Z_n) \right] 
= - \frac{1}{2} \log(2\pi \sigma^2) - \frac{1+\mu^2+\beta^2}{2\sigma^2},
\]
which implies that 
\[
\lim\limits_{n\to \infty}\E_{\theta^*}^{\pi}\left[\log f_{\theta_0}(1,Z_n) \right] 
> \lim\limits_{n\to \infty} \E_{\theta^*}^{\pi}\left[\log f_{\theta^*}(1,Z_n) \right]. 
\]
For any $\theta\in\Theta$ we have that
\begin{align*}
n^{-1} \log q_{\theta}(Z_1,\ldots,Z_{n}) 
= n^{-1} \sum\limits_{i=1}^n \log f_{\theta}(1,Z_i) 
\to\lim\limits_{n\to \infty}\E_{\theta^*}^{\pi}\left[\log f_{\theta}(1,Z_n) \right].
\end{align*}
In fact, for any closed set $C\subset \Theta$ with $\theta_0\notin C$ we have that
\[
\lim\limits_{n\to\infty} n^{-1} \log q_{\theta_0}^{\nu}(Z_1,\ldots,Z_{n}) 
> \limsup\limits_{\theta\in C} \log q_{\theta}^{\nu}(Z_1,\ldots,Z_{n})
\]
and therefore $\theta^{\rm QML}_n\to \theta_0$ a.s., see Lemma \ref{Wald lemma} 
and Theorem~\ref{Wald theorem}.
\end{example}

\subsection{Finite state space of the hidden Markov chain}
\noindent
A generalization of the consistency results of maximum likelihood 
estimation to scenarios with general state space of the 
hidden Markov chain might be of interest. 
There are mainly two reasons why we assume that $\mathcal{S}$ is finite:
\begin{enumerate}
 \item Our main motivation comes from the DHMM which models the conductance 
 levels of ion channel data with finite $\mathcal{S}$.
 \item The requirements one needs to impose get more technical. In particular,
 our conditions on the ``irreducibility and continuity of $X$'' as well as the ``well behaving HMM'' from Section~\ref{sec: struct_cond}
 become more difficult on general state spaces. It seems that the assumptions 
(A1)-(A6) of \cite{douc_consistency_2011} are sufficient, but then in the proof of Theorem~\ref{thm: main_thm} we cannot argue 
with Lemma~\ref{Lemma 13} anymore. This lemma can also be adapted to the 
more general scenario as in \cite[Lemma~13]{douc_consistency_2011}, but then involves an 
additional term.
\end{enumerate}

\subsection{Asymptotic normality of the QMLE}
\label{QMLE_asymp_nornal}
\noindent
Under additional conditions
one can obtain asymptotic normality of the MLE by applying the work of \cite{jensen_asymptotic_2011}.
The requirements to obtain this result for the QMLE are similar. Namely, 
let $\theta_{\nu,n}^{\rm QML}$ be strongly consistent, which is guaranteed under the 
assumptions of Theorem~\ref{thm: main_thm}, and assume that
\begin{itemize}
 \item the mixing condition \ref{en: M},
 \item the CLT guaranteeing condition \ref{en: CLT},
 \item as well as the uniform convergence condition \ref{en: UC},
\end{itemize}
formulated in Appendix~\ref{Ass_asymp_norm}, do hold.
Then, one can prove
\[
  \sqrt{n} \, G_n^{-1/2}\,F_n (\theta_{\nu,n}^{\rm QML} - \theta^*) \overset{\mathcal{D}}{\to}  N,
\]
where $N\sim \mathcal{N}(0,I)$, with the identity matrix $I\in \mathbb{R}^{d\times d}$,
\begin{align*}
 G_n & := \frac{1}{n} {\rm Cov}_{\theta^*}^{\pi}(S_n(\theta^*)),\\
 F_n & := -\frac{1}{n}\,\E_{\theta^*}^{\pi}\left[\left(
\frac{\partial}{\partial \theta'}S_n(\theta')\bigr\rvert_{\theta' = \theta^*}\right)^T\right],\\
S_n(\theta) & := \frac{\partial}{\partial \theta'} \log q_{\theta'}^\nu(Z_1,\dots,Z_n)\bigr\rvert_{\theta' = \theta},
\end{align*}  
and the covariance matrix of $S_n(\theta^*)$ denoted by
 ${\rm Cov}_{\theta^*}^{\pi}(S_n(\theta^*))$.
The proof of this fact is technical and follows the approach of \cite{jensen_asymptotic_2011}
by applying additional non-trivial arguments. 
The main issue of the result is the condition
\begin{equation}  
\label{eq: CLT_limit_cond}
\lim\limits_{n\rightarrow\infty}\frac{1}{\sqrt{n}}\left\vert
\E_{\theta^*}^{\pi}\left(S_n(\theta^*)\right)\right\vert_1= 0,
\end{equation}
formulated in \ref{en: CLT} in Appendix~\ref{Ass_asymp_norm} with $\vert\cdot\vert_1$ being the $\ell_1$-norm. 
It guarantees that the limiting distribution of 
$S_n$ has mean zero, which is automatically satisfied for the corresponding
quantity of the MLE. Hence, the assumptions of \cite{jensen_asymptotic_2011} for asymptotic normality of the MLE simplify 
to \ref{en: M}, \ref{en: CLT}, \ref{en: UC} of Appendix~\ref{Ass_asymp_norm}, 
where $q^\nu_\theta$ and $Z_1,\dots,Z_n$
has to be replaced by $p^\nu_\theta$ and $Y_1,\dots,Y_n$, respectively, 
without the requirement to check \eqref{eq: CLT_limit_cond}.
However, for the QMLE the crucial problem is that we are unfortunately not able to verify
\eqref{eq: CLT_limit_cond}
in the applications presented above.

 \section{Proofs and auxiliary results}
\label{sec: main_thm_proof}
\noindent
We prove some results that specify the proximity of $Y$ and $Z$.
\begin{lemma}\label{Lemma a.s. convergence}
Under the assumption formulated in \ref{en: C1}, we have
\begin{equation}
\P_{\theta}^{\nu}\left(\lim\limits_{n\rightarrow\infty}m(Z_n,Y_n)=0\right) = 1.
\end{equation}
for any $\theta\in \Theta$ and $\nu\in\mathcal{P}(S)$.
\end{lemma}
\begin{proof}
By \ref{en: C1} we obtain for any $\varepsilon>0$ that
\begin{align*}
\sum\limits_{n=1}^{\infty} \P_{\theta}^{\nu}\left(m(Z_n,Y_n)\geq \varepsilon \right)
&= \sum\limits_{n=1}^{\infty}  \sum\limits_{k=1}^{K} 
\P_{\theta}^{\nu}\left(m(Z_n,Y_n)\geq \varepsilon,X_n=k \right)\\
&= \sum\limits_{n=1}^{\infty}  \sum\limits_{k=1}^{K} 
\P_{\theta}^{\nu}\left( X_n=k \right) 
\P_{\theta}\left(m(Z_n,Y_n)\geq \varepsilon\mid X_n=k \right)\\
&\leq \sum\limits_{n=1}^{\infty} \max\limits_{k\in S} 
\P_{\theta}\left(m(Z_n,Y_n)\geq \varepsilon\mid X_n=k \right) <\infty.
\end{align*}
By the Borel-Cantelli lemma we obtain the desired almost sure convergence 
of $m(Z_n,Y_n)$ to zero.
\end{proof}

In \cite{douc_consistency_2011} the consistency of the maximum likelihood estimation 
for homogeneous HMMs under weak conditions is verified. We use the following result
of them, which verifies that the relative entropy rate exists.

\begin{thm}[{\cite[Theorem~9]{douc_consistency_2011}}]\label{Douc entropy}
Assume that the conditions \ref{en: P1} and \ref{en: H1} are satisfied. 
Then, there exists an $\ell(\theta^*)\in\R$, such that
\begin{equation}  \label{eq: ell}
\ell(\theta^*) =  \lim\limits_{n\rightarrow\infty} \E^{\pi}_{\theta^*}\left[n^{-1}\log 
q^{\pi}_{\theta^*}(Y_1,\ldots,Y_n)\right]
\end{equation}
and 
\begin{equation}
\ell(\theta^*) = \lim\limits_{n\rightarrow\infty} n^{-1} \log q^{\nu}_{\theta^*}(Y_1,\ldots,Y_n),\quad \P^{\pi}_{\theta^*}\text{-a.s.}
\end{equation}
for any probability measure $\nu\in\mathcal{P}(S)$ 
which is strictly positive if and only if $\pi$ is strictly positive.
\end{thm}
In the proof of the previous result one essentially uses the generalized 
Shannon-McMillan-Breiman theorem for stationary processes
proven by Barron et. al in \cite{barron_strong_1985}. 
Additionally, we also use a version of the generalized Shannon-McMillan-Breiman
theorem for asymptotic mean stationary processes, also proven in \cite{barron_strong_1985}. 
In the following we provide basic definitions to apply this result, for 
a detailed survey let us refer to \cite{gray_probability_2009}.
\begin{Def}
Let $(\Omega,\mathscr{F})$ be a measurable space equipped with a probability measure 
$\mathbb{Q}$ and let $T\colon \Omega \to \Omega$
be a measurable mapping. Then
\begin{itemize}
 \item $\mathbb{Q}$ is \emph{ergodic}, if for every $A\in \mathcal{I}$ either $\mathbb{Q}(A)=0$ or $\mathbb{Q}(A)=1$.
 Here $\mathcal{I}$ denotes the $\sigma$-algebra of the invariant sets, that are, the 
 sets $A\in \mathscr{F}$ satisfying
 $T^{-1}(A)=A$.
 \item $\mathbb{Q}$ is called 
 \emph{asymptotically mean stationary} (a.m.s.) if 
 there is a probability measure $\bar{\mathbb{Q}}$ on $(\Omega,\mathscr{F})$, such that
 for all $A\in \mathscr{F}$ we have
 \[
\frac{1}{n}\sum_{j=1}^n \mathbb{Q}\left(T^{-j} A\right)\underset{n\to \infty}{\longrightarrow}
\bar{\mathbb{Q}}\left(A\right).
\]
 We call $\bar{\mathbb{Q}}$ \emph{stationary mean} of $\mathbb{Q}$.
 \item a probability measure 
  $\widehat{\mathbb{Q}}$ on $(\Omega,\mathscr{F})$  
  \emph{asymptotically dominates} $\mathbb{Q}$ if for all $A\in\mathscr{F}$
with $\widehat{\mathbb{Q}}(A)=0$ holds 
\[
\lim\limits_{n\rightarrow\infty} \mathbb{Q}\left(T^{-n}A\right) = 0.
\]
\end{itemize}
\end{Def}
We need
 the following equivalence from \cite{rechard_invariant_1956}. The result follows also by
virtue of \cite[Theorem~2, Theorem~3 and the remark after the proof of Theorem~3]{gray_asymptotically_1980}.

\begin{lemma}\label{Rechard}
Let $(\Omega,\mathscr{F},\mathbb{Q})$ be a probability space 
and $T:\Omega\rightarrow\Omega$ be a measurable mapping.
Then, the following statements are equivalent:
\begin{enumerate}
\item[(i)] The probability measure $\mathbb{Q}$ is a.m.s. 
with stationary mean $\bar{\mathbb{Q}}$. 
\item[(ii)] 
There is a stationary probability measure $\widehat{\mathbb{Q}}$, which asymptotically dominates 
$\mathbb{Q}$.
\end{enumerate}
\end{lemma}

In our inhomogeneous HMM situation 
$(\Omega,\mathscr{F})$ is the space $G^{\N}$ 
equipped with the product $\sigma$-field 
$\B = \bigotimes_{i\in \mathbb{N}} \B(G)$. The transformation $T \colon G^{\N} \to G^{\N}$
is the left time shift, that is, for $A\in \B$ and $i\in\N$ we have
\begin{equation} \label{eq: left_shift}
T^{-i}(A) = \left\{(z_1,z_2,\ldots)\in G^{\N} : (z_{1+i},z_{2+i},\ldots)\in A \right\}.
 \end{equation}
Finally $\mathbb{Q}=\P_{\theta^*}^{\pi,Z}$.
In this setting we have the following result:

\begin{thm}\label{Z AMS}
Let us assume that condition \ref{en: C1} is satisfied. 
Then $\P_{\theta^*}^{\pi,Z}$ is a.m.s. with stationary mean $\P_{\theta^*}^{\pi,Y}$. 

\end{thm}
\begin{proof}
An intersection-stable generating system of the $\sigma$-algebra $\mathcal{B}$ is the union
over any finite index set $J\subset \mathbb{N}$ of cylindrical set systems
\[
 \mathcal{Z}_J := 
 \left\{
 \rho_J^{-1}(A_1\times \dots \times A_{\abs{J}})\mid A_j\in \mathcal{B}(G)\; \text{open}
 \right\},
\]
where $\rho_J\colon G^{\mathbb{N}} \to G^{\abs{J}} $ is the canonical projection to $J$, that is,
$\rho_J((a_i)_{i\in\mathbb{N}}) = (a_j)_{j\in J}$. By the uniqueness theorem
of finite measures it is sufficient to prove for an
arbitrary finite index set $J\subset \mathbb{N}$ that for any $B\in \mathcal{Z}_J$
we have
\begin{equation}  \label{eq: suff_to_prove}
 \lim_{n\to \infty} 
 \frac{1}{n} \sum_{i=1}^n \mathbb{P}_{\theta^*}^{\pi,Z}(T^{-i}(B))
 = \mathbb{P}_{\theta^*}^{\pi,Y}(B).
\end{equation}

Fix a finite index set $J=\{j_1,\ldots, j_k\}\subset \mathbb{N}$ 
and note that 
$(G^{\abs{J}},m_J)$, 
with the metric
\[
 m_J(a,b) = \sum_{j=1}^{\abs{J}} m(a_j,b_j), \quad a=(a_1,\dots,a_{\abs{J}}),\;b=(b_1,\dots,b_{\abs{J}})\in G^{\abs{J}},
\]
is a metric space. Here it is worth to mention that
the $\sigma$-algebra $\bigotimes_{j\in J}\mathcal{B}(G)$ coincides with the $\sigma$-algebra
generated by the open sets w.r.t. $m_J$.
By Lemma~\ref{Lemma a.s. convergence} we obtain
\begin{equation} \label{eq: as_zero}
\P_{\theta^*}^{\pi}\left(\lim_{i\to \infty}m_J
\left((Y_{i+j_1},\ldots,Y_{i+j_k}),(Z_{i+j_1},\ldots,Z_{i+j_k})\right)=0\right)=1.
 \end{equation}

Let $h:G^{\abs{J}}\rightarrow \R$ be a bounded, uniformly continuous function, i.e., for any 
$\varepsilon>0$ there is $\delta>0$ such that 
for all $a,b\in G^{\abs{J}}$ with $m_J(a,b)< \delta$ we have $\abs{h(a)-h(b)}<\varepsilon$. 
Then, by the stationarity of $Y$, the boundedness of $h$ and Fatou's lemma, we have
\begin{align}
\notag
 0 & \leq \liminf_{i\to \infty}
 \mathbb{E}_{\theta^*}^{\pi} [\abs{h(Z_{i+j_1},\ldots,Z_{i+j_k}) - h(Y_{j_1},\ldots,Y_{j_k})}]\\
\notag  
  & \leq \limsup_{i\to \infty}
 \mathbb{E}_{\theta^*}^{\pi} [\abs{h(Z_{i+j_1},\ldots,Z_{i+j_k}) - h(Y_{i+j_1},\ldots,Y_{i+j_k})}]\\
  & \leq  \mathbb{E}_{\theta^*}^{\pi} 
  \left[ \limsup_{i\to \infty} \abs{h(Z_{i+j_1},\ldots,Z_{i+j_k}) - h(Y_{i+j_1},\ldots,Y_{i+j_k})}\right].
\label{al: last_line}
  \end{align}
By the uniform continuity of $h$ we obtain
\[
 \lim_{i\to \infty} \abs{h(z_{i+j_1},\ldots,z_{i+j_k}) - h(y_{i+j_1},\ldots,y_{i+j_k})} = 0
\]
for all sequences $((z_{i+j_1},\ldots,z_{i+j_k}))_{i\in\mathbb{N}}$, 
$((y_{i+j_1},\ldots,y_{i+j_k}))_{i\in \mathbb{N}}\subset G^{\abs{J}}$ which satisfy
\[
\lim_{i\to\infty}m_J((z_{i+j_1},\ldots,z_{i+j_k}),(y_{i+j_1},\ldots,y_{i+j_k}))=0.
\]
Then, by using \eqref{eq: as_zero} we obtain
\[
  \mathbb{E}_{\theta^*}^{\pi} 
  \left[ \limsup_{i\to \infty} \abs{h(Z_{i+j_1},\ldots,Z_{i+j_k}) - h(Y_{i+j_1},\ldots,Y_{i+j_k})}\right]
  \leq 0,
\]
such that (by \eqref{al: last_line}) we have
\begin{equation*}
\lim\limits_{n\rightarrow\infty} \frac{1}{n} \sum\limits_{i=1}^n \E_{\theta^*}^{\pi}\left[h(Z_{i+j_1},\ldots,Z_{i+j_k})\right]
= \E_{\theta^*}^{\pi}\left[h(Y_{j_1},\ldots,Y_{j_k})\right].
\end{equation*}
Finally, by \cite[Theorem~1.2]{billingsley_convergence_1999} 
we have
for any $A \in \bigotimes_{j\in J} \mathcal{B}(G)$, 
\begin{equation*}
\lim\limits_{n\rightarrow\infty} \frac{1}{n} \sum\limits_{i=1}^n 
\P_{\theta^*}^{\pi}\left((Z_{i+j_1},\ldots,Z_{i+j_k})\in A\right)
= \P_{\theta^*}^{\pi}\left((Y_{j_1},\ldots,Y_{j_k})\in A\right),
\end{equation*}
which implies \eqref{eq: suff_to_prove} for any $B\in \mathcal{Z}_J$.
\end{proof}

Apart of the fact that we need the previous 
result to apply \cite[Theorem~3]{barron_strong_1985} it has also the following 
two useful consequences.
\begin{coro}  \label{cor: erg_thm_ams}
 Assume that condition \ref{en: C1} is satisfied. Then $\mathbb{P}^{\pi,Z}_{\theta^*}$ is ergodic.   
\end{coro}
\begin{proof}
 From \cite[Lemma~1]{leroux_maximum-likelihood_1992} it follows that $\mathbb{P}^{\pi,Y}_{\theta^*}$
 is ergodic. Then, the assertion is implied by Theorem~\ref{Z AMS} and \cite[Lemma~7.13]{gray_probability_2009}, which
 essentially states that $\mathbb{P}^{\pi,Y}_{\theta^*}$ is ergodic if and only if $\mathbb{P}^{\pi,Z}_{\theta^*}$ is ergodic.
\end{proof}
\begin{coro}
 Assume that condition \ref{en: C1} is satisfied and let $k\in \N$. Then, for any $g\colon G^k \to \mathbb{R}$
 with $\mathbb{E}_{\theta^*}^{\pi}[\abs{g(Y_1,\dots,Y_k)}]<\infty$ we have
 \[
  \lim_{n\to \infty} \frac{1}{n} \sum_{j=1}^n g(Z_{j+1},\dots,Z_{j+k}) = \mathbb{E}_{\theta^*}^{\pi}[g(Y_1,\dots,Y_k)], \quad
  \mathbb{P}_{\theta^*}^{\pi}\text{-a.s.}
 \]
\end{coro}
\begin{proof}
 By the a.m.s. property and the ergodicity of $\mathbb{P}^{\pi,Z}_{\theta^*}$ 
 the assertion is implied by \cite[Theorem~8.1.]{gray_probability_2009}. 
\end{proof}

For $z=(z_{i})_{i\in \N} \in G^{\N}$ and $k,m\in\N$ with $k< m$
we use $z_{k:m}$ to denote a segment of $z$. Specifically let
$
 z_{k:m} = (z_k,\dots,z_m).
$
Let $\lambda_k = \bigotimes_{i=1}^k \lambda$ be 
the product measure of $\lambda$ with itself, i.e., the measurable space
$(G^k,\bigotimes_{i=1}^k \mathcal{B}(G))$ is equipped with reference measure $\lambda_k$.
Now define
\[
 p_{\theta^*}^\pi(z_{1:k} \mid z_{k+1:m})
 := \frac{p_{\theta^*}^\pi (z_{1:m})}{\int_{G^k} p_{\theta^*}^{\pi}(z_{1:m}) \lambda_k(\dint z_{1:k})}.
\]
We aim to apply \cite[Theorem~3]{barron_strong_1985}. For this we need the concept
of conditional mutual information.
\begin{Def}
For $k,m,n\in\N$ define the $(k,m,n)$-conditional mutual information of $Z$ by 

\[
I^Z_{k,m}(n) \coloneqq 
\E_{\theta^*}^{\pi}\left[ \log \left(\frac{p_{\theta^*}^{\pi}(Z_{1:k}\mid Z_{k+1:k+m+n})}
{p_{\theta^*}^{\pi}(Z_{1:k}\mid Z_{k+1:k+m})} \right) \right].
\]

\end{Def}

\begin{remark}
Observe that the $(k,m,n)$-conditional mutual information of $Z$ 
coincides with the 
definition of the conditional mutual information of  
$Z_{k+m+1:k+m+n}$ and $Z_{1:k}$
given 
$Z_{k+1:k+m}$ in \cite[p.~1296]{barron_strong_1985}.
Note that
by \cite[Lemma~3]{barron_strong_1985} it is known that 
$
 I_{k,m}^Z := \lim_{n\to \infty} I^Z_{k,m}(n)
$
exists.
\end{remark}

\begin{lemma}\label{cond mutual information}
Assume that condition \ref{en: H4} is satisfied. Then, for every $k,m\in\N$ we have
$
I^Z_{k,m}:= \lim_{n\to \infty} I^Z_{k,m}(n)<\infty.
$
\end{lemma}
\begin{proof}
For $n\in \N$ we obtain
\begin{align*}
I^Z_{k,m}(n)
&  \leq 
  \E_{\theta^*}^{\pi}\left[\abs{\log p_{\theta^*}^{\pi}(Z_{1:k}\mid Z_{k+1:k+m})}\right]\\
&\qquad+ \E_{\theta^*}^{\pi}\left[\abs{\log p_{\theta^*}^{\pi}(Z_{1:k}\mid Z_{k+1:k+m+n})}\right].
\end{align*}
For $1\leq k<j$ we have by 
using $\int_{G^k} \prod_{i=1}^k f_{\theta^*,i}(s_i,z_i) \lambda_k(\dint z_{1:k})=1$
that 
\begin{align*}
     p_{\theta^*}^{\pi}(Z_{1:j})
& = \sum_{s_1,\dots,s_k\in S} \pi(s_1) 
\prod_{i=1}^{k} f_{\theta^*,i}(s_i,Z_i) \prod_{i=1}^{k-1} P_{\theta^*}(s_i,s_{i+1})\\
& \quad\times \sum_{s_{k+1},\dots,s_{j+1}\in S} P_{\theta^*}(s_k,s_{k+1}) \prod_{\ell=k+1}^{j} 
f_{\theta^*,\ell}(s_\ell,Z_\ell) P_{\theta^*}(s_\ell,s_{\ell+1})\\
&\leq \max_{x_1,\dots,x_k\in S} \prod_{i=1}^k f_{\theta^*,i}(x_i,Z_i)\, \int\limits_{G^k} p_{\theta^*}^{\pi}(z_{1:k},Z_{k+1:j})\lambda_k(\dint z_{1:k}).
\end{align*}
By \ref{en: H4} this leads to
\[
  \E_{\theta^*}^{\pi}\left[\abs{\log p_{\theta^*,k\mid j}^{\pi}(Z_{1:k}\mid Z_{k+1:j})}\right]
  \leq \max_{x_1,\dots,x_k\in S}
  \sum\limits_{i=1}^{k}\E_{\theta^*}^{\pi}
  \left[\abs{\log\left(f_{\theta^*,i}(x_i,Z_i)\right)}\right] <\infty,
\]
which gives $I_{k,m}^Z(n)<\infty$ for any $n\in \N$ and implies the assertion.
\end{proof}
The next result is a consequence of the
generalized Shannon-McMillan-Breiman
theorem for asymptotic mean stationary processes, see \cite[Theorem~3]{barron_strong_1985}.
\begin{thm}\label{Barron entropy}
Assume that the conditions \ref{en: P1}, \ref{en: C1}, \ref{en: H1} and \ref{en: H4} 
are satisfied. Then
\begin{equation*}
\lim_{n \to \infty} n^{-1} \log p_{\theta^*}^{\pi}(Z_1,\ldots,Z_n)= 
\ell(\theta^*)\quad \P_{\theta^*}^{\pi}\text{-a.s.}
\end{equation*}
(Recall that $\ell(\theta^*)$ is given by \eqref{eq: ell}.)
\end{thm}
\begin{proof}
Theorem \ref{Z AMS} shows that $\P_{\theta^*}^{\pi,Z}$ is a.m.s. with stationary mean $\P_{\theta^*}^{\pi,Y}$. 
Theorem \ref{Douc entropy} yields
\[
\lim\limits_{n\rightarrow\infty} n^{-1} \log q_{\theta^*}^{\pi}(Y_1,\ldots,Y_n) 
= \ell(\theta^*) \quad \P_{\theta^*}^{\pi}\text{-a.s.}
\] 
Lemma \ref{cond mutual information} guarantees that $I^Z_{k,m}<\infty$ for all $k,m\in\N$. 
Then, the statement follows by \cite[Theorem~3]{barron_strong_1985}.
\end{proof}

We need some auxiliary lemmas that
ensure that the ratio of 
$p^{\nu}_{\theta^*}(z_1,\ldots,z_n)$ and 
$q^{\nu}_{\theta^*}(z_1,\ldots,z_n)$ does not diverge exponentially or faster.

\begin{lemma}\label{slower than expoential}  Assume that condition \ref{en: C2} is satisfied. Then,
with $k\in \mathbb{N}$ from \ref{en: C2}, we have
\[
\limsup\limits_{n\rightarrow\infty}n^{-1}\log\left( 
\E_{\theta^*}^{\pi}\left[\prod\limits_{i=k}^n\max\limits_{s\in S} 
\frac{f_{\theta^*,i}(s,Z_i)}{f_{\theta^*}(s,Z_i)}\right]\right)\leq 0.
\]
\end{lemma}
\begin{proof}
The assertion follows from
\begin{align*}
&\limsup\limits_{n\rightarrow\infty}n^{-1}\log\left( 
\E_{\theta^*}^{\pi}\left[\prod\limits_{i=k}^n\max\limits_{s\in S} 
\frac{f_{\theta^*,i}(s,Z_i)}{f_{\theta^*}(s,Z_i)}\right]\right)\\
&=\limsup\limits_{n\rightarrow\infty}n^{-1}\log\left(\E_{\theta^*}^{\pi}\left[ 
\E_{\theta^*}^{\pi}\left[\prod\limits_{i=k}^n\max\limits_{s\in S} 
\frac{f_{\theta^*,i}(s,Z_i)}{f_{\theta^*}(s,Z_i)}\mid X_k,\ldots,X_n\right]\right]\right)\\
&=\limsup\limits_{n\rightarrow\infty}n^{-1}\log\left(\E_{\theta^*}^{\pi}\left[ 
\prod\limits_{i=k}^n\E_{\theta^*}^{\pi}\left[\max\limits_{s\in S} 
\frac{f_{\theta^*,i}(s,Z_i)}{f_{\theta^*}(s,Z_i)}\mid X_i\right]\right]\right)\\
&\leq\limsup\limits_{n\rightarrow\infty}n^{-1}\log\left(\E_{\theta^*}^{\pi}\left[ 
\prod\limits_{i=k}^n \max\limits_{s'\in S}\E_{\theta^*}^{\pi}\left[\max\limits_{s\in S} 
\frac{f_{\theta^*,i}(s,Z_i)}{f_{\theta^*}(s,Z_i)}\mid X_i=s'\right]\right]\right)\\
&=\limsup\limits_{n\rightarrow\infty}n^{-1} 
\sum\limits_{i=k}^n \max\limits_{s'\in S} \log\left( \E_{\theta^*}^{\pi}\left[\max\limits_{s\in S} 
\frac{f_{\theta^*,i}(s,Z_i)}{f_{\theta^*}(s,Z_i)}\mid X_i=s'\right]\right)\leq 0,
\end{align*}
where the last line follows from assumption \ref{en: C2}, especially \eqref{al: expect}.
\end{proof}
By the same arguments as in the previous lemma we obtain the following result.
\begin{lemma}\label{theta slower than expoential}  Assume that condition \ref{en: C3} is satisfied. Then,
for $k\in \mathbb{N}$ and $\mathcal{E}_{\theta}$ from \ref{en: C3}, we have
\[
\lim\limits_{n\rightarrow\infty}n^{-1}\log\left( 
\E_{\theta^*}^{\pi}\left[\prod\limits_{i=k}^n \sup\limits_{\theta'\in \mathcal{E}_{\theta}}   \max\limits_{s\in S} 
\frac{f_{\theta',i}(s,Z_i)}{f_{\theta'}(s,Z_i)}\right]\right)\leq 0.
\]
\end{lemma}
The next result allows us to carry the limit from Theorem~\ref{Barron entropy} over, to the case where we keep
the finite trajectory of $Z$,
but consider $q_{\theta^*}^{\nu}$ instead of $p_{\theta^*}^{\pi}$ for suitable $\nu \in \mathcal{P}(S)$.
\begin{thm}\label{thm: q_Z_well_defined_ell}
Assume that the conditions \ref{en: P1}, \ref{en: H1}, \ref{en: H4}, \ref{en: C1} 
and \ref{en: C2} are satisfied. Then
\begin{equation*}
\lim\limits_{n\rightarrow\infty} n^{-1} \log q^{\nu}_{\theta^*}(Z_1,\ldots,Z_n) 
= \ell(\theta^*)\quad \P_{\theta^*}^{\pi}\text{-a.s.}
\end{equation*}
for any probability measure $\nu\in\mathcal{P}(S)$ which is 
strictly positive if and only if $\pi$ is strictly positive.
\end{thm}

\begin{proof}
From Theorem~\ref{Barron entropy} it follows that
\begin{equation} \label{eq: zero_step}
\lim\limits_{n\rightarrow\infty} n^{-1}\log p^{\pi}_{\theta^*}(Z_1,\ldots,Z_n) 
= \ell(\theta^*) \quad \P^{\pi}_{\theta^*}\text{-a.s.}
\end{equation}
and by using \ref{en: C2} we first show
\begin{equation} \label{eq: first_step}
 \lim_{n\to \infty} n^{-1} \log q_{\theta^*}^{\pi} (Z_1,\dots,Z_n) = \ell(\theta^*) \quad \P_{\theta^*}^{\pi}\text{-a.s.}
\end{equation}
For any $\varepsilon>0$ we obtain by Markov's inequality that  
\begin{align*}
\P_{\theta^*}^{\pi}\left( n^{-1} \log\left(\frac{q_{\theta^*}^{\pi}(Z_1,\ldots,Z_n)}{p_{\theta^*}^{\pi}(Z_1,\ldots,Z_n)}\right)\geq \varepsilon\right)
&=\P_{\theta^*}^{\pi}\left( \frac{q_{\theta^*}^{\pi}(Z_1,\ldots,Z_n)}{p_{\theta^*}^{\pi}(Z_1,\ldots,Z_n)}\geq \exp(n\varepsilon)\right)\\
&\leq \exp(-n\epsilon)\cdot \E_{\theta^*}^{\pi}\left[ \frac{q_{\theta^*}^{\pi}(Z_1,\ldots,Z_n)}{p_{\theta^*}^{\pi}(Z_1,\ldots,Z_n)}\right].
\end{align*}
By the fact that
$
\E_{\theta^*}^{\pi}\left[ \frac{q_{\theta^*}^{\pi}(Z_1,\ldots,Z_n)}{p_{\theta^*}^{\pi}(Z_1,\ldots,Z_n)}\right] = 1,
$
the Borel-Cantelli Lemma implies
\begin{equation*}
\limsup\limits_{n\rightarrow\infty} n^{-1}
\log\left(
\frac{q_{\theta^*}^{\pi}(Z_1,\ldots,Z_n)}{p_{\theta^*}^{\pi}(Z_1,\ldots,Z_n)}
\right)
\leq 0 \quad \P_{\theta^*}^{\pi}\text{-a.s.}
\end{equation*}
This leads by \eqref{eq: zero_step} to 
\begin{equation}\label{entropy sup}
\limsup\limits_{n\rightarrow\infty} n^{-1} \log q_{\theta^*}^{\pi}(Z_1,\ldots,Z_n)\leq \ell(\theta^*) \quad \P_{\theta^*}^{\pi}\text{-a.s.}
\end{equation}
Observe that
\[
 \frac{p_{\theta^*}^\pi(Z_1,\dots,Z_n)}{q_{\theta^*}^\pi(Z_1,\dots,Z_n)}
 \leq \prod_{i=1}^n \max_{s\in S} \frac{f_{\theta^*,i}(s,Z_i)}{f_{\theta^*}(s,Z_i)}.
\]
Then, with the $k\in \N$ from \ref{en: C2}, in particular \eqref{eq: finite}, it follows that
\begin{align*}
&\quad \limsup\limits_{n\rightarrow\infty} n^{-1} \log\left(\frac{p_{\theta^*}^{\pi}(Z_1,\ldots,Z_n)}{q_{\theta^*}^{\pi}(Z_1,\ldots,Z_n)}\right)
\leq\limsup\limits_{n\rightarrow\infty} 
n^{-1} \log\left(\prod_{i=1}^n \max_{s\in S} \frac{f_{\theta^*,i}(s,Z_i)}{f_{\theta^*}(s,Z_i)}\right)\\
&= \limsup\limits_{n\rightarrow\infty} 
n^{-1}\left( \log\left(\prod_{i=1}^{k-1} \max_{s\in S} \frac{f_{\theta^*,i}(s,Z_i)}{f_{\theta^*}(s,Z_i)}\right)
+\log\left(\prod_{i=k}^n \max_{s\in S} \frac{f_{\theta^*,i}(s,Z_i)}{f_{\theta^*}(s,Z_i)}
\right)\right)\\
&= \limsup\limits_{n\rightarrow\infty} n^{-1}\log\left(\prod_{i=k}^n \max_{s\in S} \frac{f_{\theta^*,i}(s,Z_i)}{f_{\theta^*}(s,Z_i)}\right)\quad \P_{\theta}^{\pi}\text{-a.s.}
\end{align*}
Again, for any $\varepsilon>0$ we obtain by Markov's inequality that
\begin{align*}
&\P_{\theta^*}^{\pi}\left(n^{-1} \log\left(\prod\limits_{i=k}^n
\max\limits_{s\in S} \frac{f_{\theta^*,i}(s,Z_i)}{f_{\theta^*}(s,Z_i)}\right)
\geq \varepsilon\right)\\
& =\P_{\theta^*}^{\pi}\left( \prod\limits_{i=k}^n
\max\limits_{s\in S} \frac{f_{\theta^*,i}(s,Z_i)}{f_{\theta^*}(s,Z_i)}
\geq \exp(n\varepsilon)\right)
\leq \frac{\E_{\theta^*}^{\pi}\left[\prod\limits_{i=k}^n\max\limits_{s\in S} \frac{f_{\theta^*,i}(s,Z_i)}{f_{\theta^*}(s,Z_i)}\right]}{\exp(n\epsilon)}\\
&= \exp\left(n\left(n^{-1}\log\left( 
\E_{\theta^*}^{\pi}\left[\prod\limits_{i=k}^n\max\limits_{s\in S} 
\frac{f_{\theta^*,i}(s,Z_i)}{f_{\theta^*}(s,Z_i)}\right]\right)-\epsilon\right)\right).
\end{align*}
By Lemma \ref{slower than expoential}, the Borel-Cantelli Lemma yields 
\[
\limsup\limits_{n\rightarrow\infty} n^{-1} \log
\left(
\prod\limits_{i=k}^n
\max\limits_{s\in S} \frac{f_{\theta^*,i}(s,Z_i)}{f_{\theta^*}(s,Z_i)}
\right)\leq 0 \quad \P_{\theta^*}^{\pi}\text{-a.s.}
\]
which leads to
\[
\limsup\limits_{n\rightarrow\infty} n^{-1} \log
\left(
\frac{p_{\theta^*}^{\pi}(Z_1,\ldots,Z_n)}{q_{\theta^*}^{\pi}(Z_1,\ldots,Z_n)}
\right)\leq 0 \quad \P_{\theta^*}^{\pi}\text{-a.s.}
\]
This implies
\begin{equation}\label{entropy inf}
\liminf\limits_{n\rightarrow\infty} n^{-1} \log
\left(\frac{q_{\theta^*}^{\pi}(Z_1,\ldots,Z_n)}{p_{\theta^*}^{\pi}(Z_1,\ldots,Z_n)}\right)\geq 0 \quad \P_{\theta^*}^{\pi}\text{-a.s.}
\end{equation}
By \eqref{entropy sup} and \eqref{entropy inf} we obtain \eqref{eq: first_step}.

Next we prove the statement of the theorem using \eqref{eq: first_step}.
For any $n\in \N$ observe that

\begin{align}
\label{al: equ_qq}
\frac{q^{\pi}_{\theta^*}(Z_1,\ldots,Z_n)}{q^{\nu}_{\theta^*}(Z_1,\ldots,Z_n)} 
&= \frac{\sum\limits_{s_1,\ldots,s_{n+1}\in S} \nu(s_1) \frac{\pi(s_1)}{\nu(s_1)} 
\prod\limits_{i=1}^n f_{\theta^*}(s_i,Z_i) P_{\theta^*}(s_{i},s_{i+1})} 
{\sum\limits_{s_1,\ldots,s_{n+1}\in S} \nu(s_1) \prod\limits_{i=1}^n f_{\theta^*}(s_i,Z_i) P_{\theta^*}(s_{i},s_{i+1})}\\
\notag
&\leq \max\limits_{s\in S} \frac{\pi(s)}{\nu(s)} < \infty,
\end{align}
where the finiteness follows by the fact that $\nu$ is strictly positive if and only if $\pi$ is strictly positive.
By using \eqref{al: equ_qq} we also obtain
\begin{equation}  \label{eq: qu_low_bnd}
 \frac{q^{\pi}_{\theta^*}(Z_1,\ldots,Z_n)}{q^{\nu}_{\theta^*}(Z_1,\ldots,Z_n)}  
\geq \min\limits_{s\in S} \frac{\pi(s)}{\nu(s)} > 0.
\end{equation}
Then
\begin{align*}
& \limsup\limits_{n\rightarrow\infty} n^{-1} \log q_{\theta^*}^{\nu}(Z_1,\ldots,Z_n)\\
& = \limsup\limits_{n\rightarrow\infty}n^{-1}\left( \log \left( \frac{q^{\nu}_{\theta^*}(Z_1,\ldots,Z_n)}
{q^{\pi}_{\theta^*}(Z_1,\ldots,Z_n)}\right) + \log q^{\pi}_{\theta^*}(Z_1,\ldots,Z_n)\right)\\
&\leq  \limsup\limits_{n\rightarrow\infty}n^{-1} \left(\max\limits_{s\in S} \frac{\pi(s)}{\nu(s)} + \log q^{\pi}_{\theta^*}(Z_1,\ldots,Z_n)\right)
= \ell(\theta^*)
\end{align*}
and by \eqref{eq: qu_low_bnd} we similarly have
\[
\liminf\limits_{n\rightarrow\infty}n^{-1} \log q^{\nu}_{\theta^*}(Z_1,\ldots,Z_n)\geq \ell(\theta^*).
\]
By the previous two inequalities the assertion follows.
\end{proof}

Before we come to the proof of our main result, Theorem~\ref{thm: main_thm}, we provide 
a lemma which is essentially used and proven in \cite{douc_consistency_2011}.
In our setting the formulation and the statement 
slightly simplifies compared to \cite[Lemma~13]{douc_consistency_2011}, 
since we only consider finite state spaces $S$.

\begin{lemma}
\label{Lemma 13}
Let $\delta$ be the counting measure on $S$.
Assume that the conditions \ref{en: P1}, \ref{en: P2} and \ref{en: H1} - \ref{en: H3} are satisfied. 
Then, for any $\theta\in \Theta$ with $\theta\not \sim \theta^*$, 
there exists a natural number $n_{\theta}$ and a real number $\eta_{\theta}>0$ 
such that $B(\theta,\eta_{\theta})\subseteq \mathcal{U}_{\theta}$ 
and
\begin{equation}\label{equation 13}
\frac{1}{n_{\theta}}\E_{\theta^*}^{\pi}\left[\sup\limits_{\theta'\in B(\theta,\eta_{\theta})}\log q^{\delta}_{\theta'}(Y_1,\ldots,Y_{n_{\theta}})\right] < \ell(\theta^*).
\end{equation}
Here $B(\theta,\eta)\subseteq \Theta$ is the Euclidean ball of radius 
$\eta>0$ centered at $\theta\in\Theta$.
\end{lemma}
\begin{proof}
 The result follows straightforward from \cite[Theorem~12]{douc_consistency_2011}
 and the arguments in the proof of \cite[Lemma~13]{douc_consistency_2011}.
\end{proof}

Systematically, the proof of Theorem~1 follows the same line of 
arguments as the proof of \cite[Theorem~1]{douc_consistency_2011}.
However, let us point out that the scenario is very different:
\begin{itemize}
 \item We consider the QMLE $\theta^{\,\rm QML}_{\nu,n}$ 
 instead of the MLE. 
 \item The arguments we use heavily rely on the a.m.s. property
 of $\mathbb{P}_{\theta^*}^Z$.
\end{itemize}

\begin{proof}[\textbf{Proof of Theorem 1}]
By the standard approach to prove consistency, see Lemma \ref{Wald lemma} 
and Theorem~\ref{Wald theorem}, Theorem~\ref{thm: q_Z_well_defined_ell} 
and the fact that
\begin{align*}
q^{\nu}_{\hat{\theta}_n}(Z_1,\ldots,Z_n)\geq q^{\nu}_{\theta^*}(Z_1,\ldots,Z_n) \qquad\forall n\in\N
\end{align*}
it is sufficient
to prove for any closed set $C\subseteq \Theta$ 
with $\theta^*\not\in C$ that
\begin{equation*}
\limsup\limits_{n\rightarrow\infty} \sup\limits_{\theta'\in C}n^{-1} 
\log q_{\theta'}^{\nu}(Z_1,\ldots,Z_n) < \ell(\theta^*)\quad\P_{\theta}^{\pi}\text{-a.s.}
\end{equation*}
Note that, with $\eta_{\theta}$ defined in Lemma~\ref{Lemma 13},
the set $\{B(\theta,\eta_{\theta}),\theta\in C\}$ is a cover of $C$. 
As $\Theta$ is compact, $C$ is also compact and
thus admits a finite subcover $\{B(\theta_i,\eta_{\theta_i}),\theta_i \in C, i=1,\ldots,N\}$. 
Hence it is enough to verify
\begin{equation} \label{eq: to show}
\limsup\limits_{n\rightarrow\infty} \sup\limits_{\theta'\in B(\theta,\eta_{\theta})\cap C}
n^{-1} \log q_{\theta'}^{\nu}(Z_1,\ldots,Z_n) < \ell(\theta^*)\quad\P_{\theta}^{\pi}\text{-a.s.}
\end{equation}
for any $\theta\not\sim\theta^*$.

Let us fix $\theta\not \sim \theta^*$ and let $\eta_\theta$ as well as $n_\theta$
as in Lemma~\ref{Lemma 13}. Observe that 
for 
any $\theta' \in \Theta$ 
and 
any $1\leq m \leq n$ we have
\begin{align}
\label{leq 1}
 q_{\theta'}^\nu(z_{1},\dots,z_n) 
 & \leq q_{{\theta'}}^\nu(z_{1},\dots,z_{m-1})\, q_{\theta'}^\delta(z_{m},\dots,z_{n}),\\
\label{leq 2}
 q_{\theta'}^\delta(z_{1},\dots,z_n) & \leq q_{{\theta'}}^\delta(z_{1},\dots,z_{m-1}) 
 q_{\theta'}^\delta(z_{m},\dots,z_n),
\end{align}
and define $g^*_{{\theta'},m,n}(z_m,\dots,z_n):= \prod_{i=m}^n \max_{s\in S} f_{\theta'}(s,z_i)$
as well as $i(n):= \lfloor n/n_\theta \rfloor$.

By using those definitions, and by (\ref{leq 1}) and (\ref{leq 2}) we obtain
for sufficiently large $n\in \N$ that
\begin{align*}
 \ell_{\nu,n}^{\,\rm Q}(\theta') 
&\leq \frac{1}{n_{\theta}}\sum\limits_{r=1}^{n_{\theta}} \ell_{\nu,r}^{\,\rm Q}(\theta') + \log q^{\delta}_{\theta'}(Z_{r+1},\ldots,Z_{n})\\
&\leq \frac{1}{n_{\theta}} \sum\limits_{r=1}^{n_{\theta}} 
\log g^*_{\theta',1,r}(Z_{1},\ldots,Z_{r})\\
&\quad+\frac{1}{n_{\theta}} \sum\limits_{r=1}^{n_{\theta}}\sum\limits_{k=1}^{i(n)-1} \log q^{\delta}_{\theta'}(Z_{n_{\theta}(k-1)+r+1},\ldots,Z_{n_{\theta}k+r})\\
&\qquad+ \frac{1}{n_{\theta}} \sum\limits_{r=1}^{n_{\theta}} 
\log g^{*}_{\theta',n_{\theta}(i(n)-1)+r+1,n}(Z_{n_{\theta}(i(n)-1)+r+1},\ldots,Z_{n}) \\
&= 
\frac{1}{n_{\theta}} \sum\limits_{r=1}^{n_{\theta}} \log g^{*}_{\theta',1,r}(Z_{1},\ldots,Z_{r})\\
&\quad+\frac{1}{n_{\theta}} \sum\limits_{r=1}^{n_{\theta}(i(n)-1)} \log q^{\delta}_{\theta'}(Z_{r+1},\ldots,Z_{n_{\theta}+r})\\
&\qquad+\frac{1}{n_{\theta}} \sum\limits_{r=1}^{n_{\theta}} \sum\limits_{k=n_{\theta}(i(n)-1)+r+1}^{n}  
\max\limits_{s\in S}\log f_{\theta'}(s,Z_k).
\end{align*}
Observe that for $1\leq r \leq n_{\theta}$ holds $n_{\theta}(i(n)-1)+r\geq n - 2n_{\theta}$.
Hence we can further estimate the last average and obtain
\begin{align}
\sup\limits_{\theta'\in B(\theta,\eta_{\theta})\cap C}\md{ \ell_{\nu,n}^{\,\rm Q}}(\theta')
&\leq 
\frac{1}{n_{\theta}} \sum\limits_{r=1}^{n_{\theta}} 
\sup\limits_{\theta'\in B(\theta,\eta_{\theta})\cap C} \log g^{*}_{\theta',1,r}(Z_{1},\ldots,Z_{r})\notag\\
&\quad+\frac{1}{n_{\theta}}   \sum\limits_{r=1}^{n_{\theta}(i(n)-1)} \sup\limits_{\theta'\in B(\theta,\eta_{\theta})\cap C} \log q^{\delta}_{\theta'}(Z_{r+1},\ldots,Z_{n_{\theta}+r})\notag\\ 
\notag
&\qquad+ \sum\limits_{k=n - 2n_{\theta}+1}^{n} \sup\limits_{\theta'\in B(\theta,\eta_{\theta})\cap C} \max\limits_{s\in S}\log\left( f_{\theta'}(s,Z_k)\right)^{+}.
\end{align}
We multiply both sides of the previous inequality by $n^{-1}$ 
and consider the limit $n\to\infty$
of each sum on the right-hand side. In particular we show 
that the right-hand side is smaller than $\ell(\theta^*)$ which verifies \eqref{eq: to show}. 

\textbf{To the first sum:} 
By the fact that 
$
\int_{G} f_{\theta'}(s,z) \lambda(\dint z) = 1,
$
for any $s\in S$
we conclude $\lambda(\left\{z\in G : f_{\theta}(s,z) = \infty \right\})=0$. 
Hence
\[
\P_{\theta^*}^{\pi}\left(f_{\theta'}(s,Z_i)=\infty\right) = 0,
\]
and \ref{en: H3} implies
\[
\P_{\theta^*}^{\pi}\left(\sup\limits_{\theta'\in B(\theta,\eta_{\theta})\cap C}  
\log g^*_{\theta',1,r}(Z_1,\ldots,Z_r)=\infty\right) = 0\qquad\forall r\in\N.
\]
This leads to
\begin{equation*}
\lim\limits_{n\rightarrow\infty}\frac{1}{n}\frac{1}{n_{\theta}} 
\sum\limits_{r=1}^{n_{\theta}} \sup\limits_{\theta'\in B(\theta,\eta_{\theta})\cap C} \log g^{*}_{\theta',1,r}(Z_{1},\ldots,Z_{r}) = 0 \qquad \P_{\theta^*}	^{\pi}\text{-a.s.}
\end{equation*}

\textbf{To the second sum:} By the fact that 
$i(n)/n\rightarrow n_{\theta}^{-1}$ as $n\rightarrow\infty$,

Lemma \ref{Lemma 13} and Corollary~\ref{cor: erg_thm_ams} we obtain

\begin{align*}
&\lim\limits_{n\rightarrow\infty}\frac{1}{n}\frac{1}{n_{\theta}}   \sum\limits_{r=1}^{n_{\theta}(i(n)-1)} \sup\limits_{\theta'\in B(\theta,\eta_{\theta})\cap C} \log q^{\delta}_{\theta'}(Z_{r+1},\ldots,Z_{n_{\theta}+r})\notag\\
&\qquad\qquad\qquad\qquad= \frac{1}{n_{\theta}}\E_{\theta^*}^{\pi}\left[\sup\limits_{\theta'\in B(\theta,\eta_{\theta})\cap C} \log p^{\delta}_{\theta'}(Y_{1},\ldots,Y_{n_{\theta}})\right]
<\ell(\theta^*).
\end{align*}

\textbf{To the third sum:}
By assumption \ref{en: H2} it follows that
\[
\E_{\theta^*}^{\pi}\left[\sup\limits_{\theta'\in \mathcal{U}_{\theta}}\max\limits_{s\in S} \left(\log f_{\theta}(s,Y_1)\right)^+\right]\leq 
\sum\limits_{s\in S}\E_{\theta^*}^{\pi}\left[\sup\limits_{\theta'\in \mathcal{U}_{\theta}} \left(\log f_{\theta}(s,Y_1)\right)^+\right]<\infty
\]
and by Corollary~\ref{cor: erg_thm_ams} we have
\[
\lim\limits_{n\rightarrow\infty}\frac{1}{n}\sum\limits_{k=1}^n \sup\limits_{\theta'\in B(\theta,\eta_{\theta})} \max\limits_{s\in S} \log (f_{\theta'}(s,Z_k))^+
= \E_{\theta^*}^{\pi}\left[\sup\limits_{\theta'\in B(\theta,\eta_{\theta})}\max\limits_{s\in S} \left(\log f_{\theta}(s,Y_1)\right)^+\right].
\]
Hence
\begin{equation*}
\lim\limits_{n\rightarrow\infty}\frac{1}{n}\sum\limits_{k=n-2n_{\theta}+1}^n \sup\limits_{\theta'\in B(\theta,\eta_{\theta})} \max\limits_{s\in S} \log (f_{\theta'}(s,Z_k))^+
= 0\qquad \P_{\theta^*}	^{\pi}\text{-a.s.}
\end{equation*}
and the proof is complete 
\end{proof}

As a consequence of the proof of Theorem~\ref{thm: main_thm} we 
are able to prove consistency for the MLE
under condition \ref{en: C3}. 

\begin{proof}[\textbf{Proof of Corollary~\ref{lemma: main lemma}}]
We use the same strategy as in the proof of Theorem \ref{thm: main_thm}. By Theorem \ref{Barron entropy} it follows
that
\begin{equation*}
\lim_{n \to \infty} n^{-1} \log p_{\theta^*}^{\pi}(Z_1,\ldots,Z_n)= 
\ell(\theta^*)\quad \P_{\theta^*}^{\pi}\text{-a.s.}
\end{equation*}
For $\theta\not \sim \theta^*$, we chose 
$\kappa_{\theta}\leq \eta_{\theta}$, where $\eta_{\theta}$ 
is defined in Lemma \ref{Lemma 13},
such that $B(\theta,\kappa_{\theta})\subset \mathcal{E}_{\theta}$. 
As explained in the proof of Theorem \ref{thm: main_thm},
 it is sufficient to verify for any closed set $C\subseteq \Theta$
 with $\theta^*\not\in C$ that
\begin{equation} \label{eq: ML to show}
\limsup\limits_{n\rightarrow\infty} \sup\limits_{\theta'\in B(\theta,\kappa_{\theta})\cap C}
n^{-1} \log p_{\theta'}^{\nu}(Z_1,\ldots,Z_n) < \ell(\theta^*)\quad\P_{\theta}^{\pi}\text{-a.s.}
\end{equation}
With $k\in \N$ from condition \ref{en: C3} 
we obtain by using \eqref{eq: theta finite} that
\begin{align*}
&\limsup\limits_{n\rightarrow\infty} \sup\limits_{\theta'\in B(\theta,\kappa_{\theta})\cap C}
n^{-1} \log\left( \frac{p_{\theta'}^{\nu}(Z_1,\ldots,Z_n)}{q_{\theta'}^{\nu}(Z_1,\ldots,Z_n)}\right)\\
&\leq \limsup\limits_{n\rightarrow\infty} \sup\limits_{\theta'\in B(\theta,\kappa_{\theta})\cap C} n^{-1}\log\left(
\prod\limits_{i=1}^n \max\limits_{s\in S}\frac{f_{\theta',i}(s,Z_i)}{f_{\theta'}(s,Z_i)}\right)\\
&=  \limsup\limits_{n\rightarrow\infty} \sup\limits_{\theta'\in B(\theta,\kappa_{\theta})\cap C} n^{-1}\log\left(
\prod\limits_{i=k}^n \max\limits_{s\in S}\frac{f_{\theta',i}(s,Z_i)}{f_{\theta'}(s,Z_i)}\right)\\
&\leq  \limsup\limits_{n\rightarrow\infty} n^{-1}\log\left(
\prod\limits_{i=k}^n  \sup\limits_{\theta'\in B(\theta,\kappa_{\theta})\cap C} \max\limits_{s\in S}\frac{f_{\theta',i}(s,Z_i)}{f_{\theta'}(s,Z_i)}\right).
\end{align*}
By the same arguments as for proving \eqref{entropy inf} 
in the proof of Theorem~\ref{thm: q_Z_well_defined_ell}
we get that
\begin{align*}
\P_{\theta^*}^{\pi}\left(n^{-1} \log\left(\prod\limits_{i=k}^n  
\sup\limits_{\theta'\in B(\theta,\kappa_{\theta})\cap C} 
\max\limits_{s\in S}\frac{f_{\theta',i}(s,Z_i) }{f_{\theta'}(s,Z_i)}\right)\geq 
\varepsilon\right)\leq \exp\left(n(c_n -\varepsilon)\right),
\end{align*}
with
\[
c_n :=   \limsup\limits_{n\rightarrow\infty} n^{-1} \log\left(\E_{\theta^*}^{\pi}\left[ \prod\limits_{i=k}^n  \sup\limits_{\theta'\in B(\theta,\kappa_{\theta})\cap C}
  \max\limits_{s\in S}  \frac{f_{\theta',i}(s,Z_i)}{f_{\theta'}(s,Z_i)} \right] \right).
\]
Assumption \ref{en: C3}, 
in particular Lemma~\ref{theta slower than expoential}, 
and the Borel Cantelli lemma implies that
\[
\P_{\theta^*}^{\pi}\left( \limsup\limits_{n\rightarrow\infty} 
\sup\limits_{\theta'\in B(\theta,\kappa_{\theta})\cap C}
n^{-1} \log\left( \frac{p_{\theta'}^{\nu}(Z_1,\ldots,Z_n)}
{q_{\theta'}^{\nu}(Z_1,\ldots,Z_n)}\right) \leq 0 \right) = 1.
\]
Similarly, it follows that 
\[
\P_{\theta^*}^{\pi}\left( \limsup\limits_{n\rightarrow\infty} \sup\limits_{\theta'\in B(\theta,\kappa_{\theta})\cap C}
n^{-1} \log\left( \frac{q_{\theta'}^{\nu}(Z_1,\ldots,Z_n)}{p_{\theta'}^{\nu}(Z_1,\ldots,Z_n)}\right) \leq 0 \right) = 1,
\]
which implies
\begin{align*}
& \limsup\limits_{n\rightarrow\infty} \sup\limits_{\theta'\in B(\theta,\kappa_{\theta})\cap C}
n^{-1} \log p_{\theta'}^{\nu}(Z_1,\ldots,Z_n) \\
& =
\limsup\limits_{n\rightarrow\infty} \sup\limits_{\theta'\in B(\theta,\kappa_{\theta})\cap C}
n^{-1} \log q_{\theta'}^{\nu}(Z_1,\ldots,Z_n).
\end{align*}
Finally the assertion follows from (\ref{eq: to show}).
\end{proof}

{\bf Acknowledgments.}
We thank the referees for their careful reading of the manuscript and their comments.
Manuel Diehn and Axel Munk gratefully acknowledge support of the CRC 803 Project C2.
Daniel Rudolf gratefully acknowledges support of the 
Felix-Bernstein-Institute for Mathematical Statistics in the Biosciences 
(Volkswagen Foundation) and the Campus laboratory AIMS.

\appendix
\section{Strong consistency}\label{Appendix}
\noindent
We follow the classical  the approach of Wald, 
see \cite{wald_note_1949}, adopted to quasi likelihood estimation. 
 Let $(\Omega,\mathscr{F},\mathbb{P})$ be a probability 
space and $(G,\mathscr{G})$ be a measurable space. 
Assume that $\Theta \subseteq \mathbb{R}^d$ and 
let $\abs{\cdot}$ be the $d$-dimensional Euclidean norm.

\begin{thm}[Strong consistency] \label{Wald theorem}Let $(W_n)_{n\in\N}$ be a sequence of random variables
mapping from $(\Omega,\mathscr{F},\P)$ to $(G,\mathscr{G})$. For any $n\in \N$ let
$h_n:\Theta \times G^n\rightarrow [0,\infty)$ be a measurable function. Assume that there exists an element $\theta^*\in \Theta$
such that for any closed $C\subset\Theta$ with $\theta^*\not\in C$ and all $n\in\N$, we have
\begin{equation}\label{Wald 1}
 \lim\limits_{n\rightarrow\infty}\sup\limits_{\theta\in C} \frac{h_n(\theta,W_1,\ldots,W_n)}
 {h_n(\theta^*,W_1,\ldots,W_n)}=0 \quad \mathbb{P}\text{-a.s.}
\end{equation}
Let $(\hat{\theta}_n)_{n\in\N}$ be a sequence of random variables mapping from $(\Omega,\mathscr{F},\P)$ 
to $\Theta$ such that 
\begin{equation}\label{Wald 2}
\exists c>0\;\, \&\;\, n_0\in \N\quad\forall\, n\geq n_0:
 \quad
 \frac{h_n(\hat \theta_n,W_1,\dots,W_n)}{h_n(\theta^*,W_1,\dots,W_n)} \geq c,\quad \mathbb{P}\text{-a.s.}
\end{equation} 
Then 
\[
\lim\limits_{n\rightarrow\infty} \abs{\hat{\theta}_n-\theta^*} = 0 \quad \P\text{-a.s.}
\]
\end{thm}
\begin{proof}
For arbitrary $\varepsilon>0$ define
\begin{align*}
 A_\varepsilon^{(1)} 
 &:= \left\{\omega\in \Omega\colon \limsup\limits_{n\rightarrow\infty}\, \abs{\hat{\theta}_n(\omega)-\theta^*}>\varepsilon 
      \right\},\\
 A_\varepsilon^{(2)}
 &:= \left\{\omega\in \Omega\colon 
 \limsup\limits_{n\rightarrow\infty}\sup_{\theta:\,\abs{\theta-\theta^*}\geq \varepsilon} 
 \frac{h_n(\theta,W_1(\omega),\ldots,W_n(\omega))}{h_n(\hat{\theta}_n(\omega),W_1(\omega),\ldots,W_n(\omega))}\geq1
      \right\},\\
 A_{\varepsilon}^{(3)}
 & := \left\{ \omega\in \Omega\colon 
 \limsup\limits_{n\rightarrow\infty}\sup_{\theta:\,\abs{\theta-\theta^*}\geq \varepsilon} 
 \frac{h_n(\theta,W_1(\omega),\ldots,W_n(\omega))}{h_n(\theta^*,W_1(\omega),\ldots,W_n(\omega))}\geq c
 \right\}.
\end{align*}
Note that $A_\varepsilon^{(1)}\subseteq A_\varepsilon^{(2)} \subseteq A_\varepsilon^{(3)}$,
where the last inclusion follows by \eqref{Wald 2}.
Hence, by \eqref{Wald 1} we have $\mathbb{P}(A_\varepsilon^{(3)})=0$ so that
\[
 \mathbb{P}(A_{\varepsilon}^{(1)}) 
 = \P\left(\limsup\limits_{n\to \infty} \;\abs{\hat{\theta}_n-\theta^*}>\varepsilon\right) = 0,
\]
which implies the assertion.
\end{proof}
The following lemma is useful to verify condition \eqref{Wald 1}.
\begin{lemma}\label{Wald lemma}
Let $(W_n)_{n\in\N}$ be a sequence of random variables
mapping from $(\Omega,\mathcal{F},\P)$ to $(G,\mathscr{G})$ and, as in Theorem~\ref{Wald theorem},
for any $n\in \N$ let
$h_n:\Theta \times G^n\rightarrow [0,\infty)$ be a measurable function. 
Assume that there is an element $\theta^*\in \Theta$ such that for any closed $C\subset \Theta$
with $\theta^*\not \in C$ we have
\begin{equation} \label{eq: appendix_wald_lemma}
 \limsup\limits_{n\to \infty} \sup_{\theta \in C}
 \frac{1}{n} \log h_n(\theta,W_1,\dots,W_n)
 < \lim_{n\to \infty} \frac{1}{n} \log h_n(\theta^*,W_1,\dots,W_n)\quad\P\text{-a.s.}
\end{equation}
provided that the limit on the right hand-side exists.
Then condition \eqref{Wald 1} is satisfied.
\begin{proof}

Obviously \eqref{eq: appendix_wald_lemma} implies
\[
 \log\left(\limsup_{n\to \infty} 
 \sup_{\theta\in C} \left[ \frac{h_n(\theta,W_1,\dots,W_n)}{h_n(\theta^*,W_1,\dots,W_n)} 
 \right]^{1/n}
 \right)<0.
\]
This leads to
\[
 \limsup_{n\to \infty} 
 \sup_{\theta\in C} \left[ \frac{h_n(\theta,W_1,\dots,W_n)}{h_n(\theta^*,W_1,\dots,W_n)} 
 \right]^{1/n}<1
\]
from which \eqref{Wald 1} follows.
\end{proof}
\end{lemma}
\section{Assumptions for asymptotic normality}\label{Ass_asymp_norm}
\noindent
For the MLE to achieve a statement about asymptotic
normality one can apply the theory for $M$-estimators developed by Jensen in \cite{jensen_asymptotic_2011}.
Before we are able to formulate assumptions which lead to asymptotic normality of
$\theta_{\nu,n}^{{\rm QML}}$ we need some further notations.
Recall that $B(\theta^*,\delta)$ is the Euclidean ball of radius $\delta>0$ centered at $\theta^*\in \Theta$.
For $v\in \R^d$ let $\vert v \vert_1$ be the $\ell_1$-norm. 
Consider a sequence of functions $(a_i)_{i\in\mathbb{N}}$ with 
$a_i:\Theta\times S\times S\times G\rightarrow \mathbb{R}$. We say that 
$(a_i)_{i\in\N}$ belongs to the class $C_k$ if there exist a sequence of functions $(a^0_i)_{i\in\mathbb{N}}$,
with $a^0_i:G\rightarrow [0,\infty)$,
a constant $\delta_0>0$ and a constant $K<\infty$ such that for all $i\in\N$, 
\[
\sup_{s_1,s_2\in S,\,\theta\in B(\theta^*,\delta_0)} \abs{a_i(\theta,s_1,s_2,z)}\leq a_i^0(z) \quad \forall z\in G
\quad\text{and}\quad 
\E_{\theta^*}^{\pi}\left[ a_i^0(Z_i)^k\right]\leq K.
\] 
Furthermore, $(a_i)_{i\in\N}$ belongs to the class $C_{k,m}$ if $(a_i)_{i\in\N}\in C_k$, 
there exist a sequence of functions $(\bar{a}_i)_{i\in\mathbb{N}}$, with $\bar{a}_i:G\rightarrow [0,\infty)$, 
and $\delta_0>0$
such that for all $\theta\in B(\theta^*,\delta_0)$, for all $s_1,s_2\in S$ and for all $i\in\N$,
\[
\abs{a_i(\theta,s_1,s_2,z)-a_i(\theta^*,s_1,s_2,z)} \leq \abs{\theta-\theta^*}\bar{a}_i(z)
\quad \forall z\in G
\quad\text{and}
\quad \E_{\theta^*}^{\pi}\left[\bar{a}_i(Z_i)^m\right]\leq K.
\]
For a positive semi-definite, symmetric matrix $A\in \mathbb{R}^{d\times d}$ let $\lambda_{\min}(A)$
to be the smallest eigenvalue of $A$.
For any $\theta\in \Theta$ define the gradient
\[
 S_n(\theta) := \frac{\partial}{\partial \theta'} \log q_{\theta'}^\nu(Z_1,\dots,Z_n)\bigr\rvert_{\theta' = \theta} 
\]
and note that with random vectors
\begin{equation}  \label{eq: psi_i}
 \psi_i(\theta) := 
 \begin{cases}
  \frac{\partial}{\partial \theta'} \log (P_{\theta'}(X_{i-1},X_i) f_{\theta'}(X_i,Z_i))\bigr\rvert_{\theta' = \theta}, & i\geq 2,\\
  \frac{\partial}{\partial \theta'} \log (\nu(X_1) f_{\theta'}(X_1,Z_1))\bigr\rvert_{\theta' = \theta}, & i=1,
 \end{cases}
\end{equation}
a simple calculation reveals $S_n(\theta) = \sum_{i=1}^n \mathbb{E}_\theta^\nu(\psi_i(\theta)\mid Z_1,\dots,Z_n)$.
The following three conditions are needed to adapt the proof of the asymptotic normality
for the MLE of \cite{jensen_asymptotic_2011} to the QMLE:
\subsubsection*{Mixing}
\begin{enumerate}[label=(M)]
\label{sec: Mixing}
\item \label{en: M}There is a constant $c_0>0$ such that
\[
c_0 \leq P_{\theta^*}(s_1,s_2) \quad\forall s_1,s_2\in S.
\]
\end{enumerate}
\subsubsection*{Central Limit Theorem}
\begin{enumerate}[label=(CLT)]
\item \label{en: CLT} 
Assume that 
\begin{equation*}
\lim\limits_{n\rightarrow\infty}\frac{1}{\sqrt{n}}\left\vert
\E_{\theta^*}^{\pi}\left(S_n(\theta^*)\right)\right\vert_1= 0,
\end{equation*}
and that
$(\psi_i
)_{i\in\N}\in C_3$. Furthermore, there exist constants $c_1>0$ and 
$n_0\in \N$ such that for $n\geq n_0$ holds
\[
\lambda_{\min}\left(\frac{1}{n}{\rm Cov}_{\theta^*}^{\pi}(S_n(\theta^*))\right) \geq c_1,
\]
where ${\rm Cov}_{\theta^*}^{\pi}(S_n(\theta^*))$ denotes the covariance matrix of $S_n(\theta^*)$.
\end{enumerate}
\subsubsection*{Uniform Convergence}
\begin{enumerate}[label=(UC)]
\item \label{en: UC} Let $F_n\in \mathbb{R}^{d\times d}$ be defined by
\[
F_n := -\frac{1}{n}\,\E_{\theta^*}^{\pi}\left[\left(
\frac{\partial}{\partial \theta'}S_n(\theta')\bigr\rvert_{\theta' = \theta^*}\right)^T\right].
\]
 Assume that there exist constants $c_2>0$ and $n_0\in\N$ such that for $n\geq n_0$
 holds $
\lambda_{\min}\left(F_n\right) \geq c_2.
  $
Furthermore, assume that $(\psi_i)_{i\in\N}$ is of class $C_4$ and 
for any $r=1,\dots,d$ we have that $(\partial \psi_i/\partial \theta_r)_{i\in \N}$ is of class $C_{3,1}$.
\end{enumerate}

\bibliographystyle{elsarticle-num}

\end{document}